\documentclass[numbook, envcountsect, envcountsame, envcountreset, runningheads, smallextended]{svjour3}

\journalname{Finance and Stochastics}

\usepackage[centertags,intlimits]{amsmath}                     
\usepackage{amsfonts}

\providecommand{\abs}[1]{\lvert #1\rvert}
\providecommand{\norm}[1]{\lVert #1\rVert}

\newcommand{\nc}{\newcommand}
\nc{\vb}{\mathbf{v}}
\nc{\bx}{\mathbf{x}}
\nc{\by}{\mathbf{y}}
\nc{\bz}{\mathbf{z}}
\nc{\bu}{\mathbf{u}}
\nc{\bv}{\mathbf{v}}
\nc{\ba}{\mathbf{a}}
\nc{\bs}{\mathbf{s}}
\nc{\bq}{\mathbf{q}}
\nc{\bd}{\mathbf{d}}
\nc{\bb}{\mathbf{b}}
\nc{\bc}{\mathbf{c}}
\nc{\bi}{\mathbf{i}}
\nc{\bfr}{\mathbf{r}}
\nc{\bA}{\mathbf{A}}
\nc{\R}{\mathbb R}
\nc{\N}{\mathbb N}
\nc{\C}{\mathbb C}
\nc{\D}{\mathbb D}
\nc{\Z}{\mathbb Z}
\nc{\F}{\mathbf F}
\nc{\bbS}{\mathbb S}
\nc{\B}{\cal B}
\nc{\br}{\bigr}
\nc{\bl}{\bigl}
\nc{\Bl}{\Bigl}
\nc{\Br}{\Bigr}
\nc{\ind}{\mathbf{1}}
\nc{\bP}{\mathbf{P}}

 \def\JELname{{\bfseries JEL Classification}\enspace}
      \def\JEL#1{\par\addvspace\medskipamount{\rightskip=0pt plus1cm
      \def\and{\ifhmode\unskip\nobreak\fi\ $\cdot$
      }\noindent\JELname\ignorespaces#1\par}}
\begin{document}

\title{On minimizing a portfolio's shortfall probability
}
\author{Anatolii A. Puhalskii\\         \and\\
Michael Jay Stutzer         
}

\institute{Anatolii A.  Puhalskii \at
 Institute for Problems in Information
Transmission\\
and University of Colorado\\
              \email{anatolii.puhalskii@ucdenver.edu}           
           \and
Michael Jay Stutzer \at
University of Colorado\\\email{michael.stutzer@Colorado.EDU}
}

\date{Received:\hspace{1.cm}           / Accepted:\hspace{1.cm} }


\sloppy
\maketitle

\begin{abstract}
We obtain a lower  asymptotic bound on the
decay rate of the probability of a portfolio's
underperformance  against a benchmark 
over a large time horizon.
A diffusion model of the asset prices is assumed,
with the mean returns and volatilities 
being represented by possibly nonlinear functions 
  of an economic factor.  The bound is tight, more specifically,
we are able to produce 
     $\epsilon$--asymptotically optimal 
 portfolios.


 \JEL{C6}
\end{abstract}

\section{Introduction}
\label{sec:introduction}
As noted by Roll \cite[p.13]{kn:roll2}
`` Today's professional money manager is often judged by total return performance relative to a prespecified benchmark, usually a broadly diversified index of assets.''   He argues that ``This is a sensible approach because the sponsor's most direct alternative to an active manager is an index fund matching the benchmark.'' 
A typical example, of more than just professional interest to academic readers, is the following statement by the TIAA--CREF Trust Company:
\begin{quote}
Different accounts have different benchmarks based on the client's
overall objectives... Accounts for clients who have growth objectives
with an emphasis on equities will be benchmarked heavily toward the
appropriate equity index -- typically the S\&P 500 index -- whereas an
account for a client whose main objective is income and safety of
principal will be measured against a more balanced weighting of the
S\&P 500 and the Lehman Corporate/Government 
Bond Index  \cite[p.3]{kn:tiaa}.
\end{quote} 
         How should plan sponsors and the investors they represent evaluate the performance of a fund like this?  Nobel Laureate William Sharpe \cite[p.32]{kn:sharpe3} asserts that 
\begin{quote}
The key information an investor needs to evaluate a mutual fund is (i) the fund's likely future exposures to movements in major asset classes, (ii) the likely added (or subtracted) return over and above a benchmark with similar exposures, and (iii) the likely risk vis-\'{a}-vis the benchmark.  
\end{quote}
This paper will focus on Sharpe's aforementioned point (iii):  how to
choose a portfolio that minimizes its shortfall risk vis-\'{a}-vis an
exogenous benchmark.   Sensoy \cite[p.26]{kn:sensoy} notes that ``the
vast majority of actively
managed, diversified U.S. equity funds use a S\&P or Russell benchmark index that is defined on size or value/growth dimensions.'' 
Evidence that a fund's investors want to outperform its designated benchmark, i.e. not fall short of doing so,  was also found by Sensoy (op.cit., p. 38), who concluded:

\begin{quote} 
Performance relative to the specified benchmark, especially above the benchmark, is a significant determinant of a fund's subsequent cash inflows, even controlling for performance measures that better capture the fund's style.
\end{quote} 
Because ``mutual funds generally receive a fixed percentage of assets under management as a fee'' [Sensoy (op.cit.,p.33)],  the fees received by mutual fund management increase when the fund outperforms its designated benchmark and decrease when it doesn't.  Hence a fund's managers have strong incentives to minimize the shortfall risk vis--\'{a}--vis the fund's  exogenous benchmark.  
But to study how that objective might be achieved, one first must rigorously define ``the shortfall risk vis--\'{a}--vis the fund's exogenous benchmark''.  One could, for example, fix a specific investment horizon length {\em T}, and search for a portfolio  that minimizes the probability that its value at {\em T} will be less than the value of the designated benchmark at {\em T}. But how should this horizon length be determined?  Does a typical fund investor have a specific horizon length in mind?  Even if an investor does have a specific horizon length in mind, won't a mutual fund wind up with investors with different horizon lengths in mind?   
In light of these unanswered questions, and with an eye toward endowment, pension, or retirement investors who are interested in relatively ``long--run'' returns, we will use large deviation asymptotics to characterize portfolios with minimum feasible ``long--run'' shortfall probabilities.  While this will take the form of a continuous time, dynamic optimization problem over an infinite horizon,  that does not imply that the optimal portfolio will have bad performance over fixed, shorter horizons that may interest other investors.

The problem of optimizing the probability of 
underperformance of a financial portfolio 
 over an infinite time horizon by
using large deviation asymptotics has been studied by a number of
authors.
Discrete time setups were considered in
   Stutzer \cite{stu00,stu03b,Stu03a,stu04}.
Continuous--time models were treated in
Hata,  Nagai,  and Sheu \cite{Hat10} and Nagai \cite{Nag12}.
More specifically, the latter authors
 concern themselves with diffusion models of
asset prices.
 Hata,  Nagai,  and Sheu
\cite{Hat10}  assume that the mean returns and volatilities of
 the security prices 
are affine functions of the economic factor
and that the economic factor is represented by a Gaussian
process,  that the risk--free interest rate does not depend
on the economic factor and that no benchmark is involved.
In Nagai \cite{Nag12}, a nonlinear model is considered with the
risk--free  asset as  the benchmark.
Two kinds of optimal portfolios are obtained in 
Hata,  Nagai,  and Sheu \cite{Hat10} and in Nagai \cite{Nag12}. The first
one is a time--dependent portfolio. At first, one has to choose 
investment horizon $T$ and solve an optimal control problem on
$[0,T]$\,. When horizon $T$ changes, a different optimal control
problem has to be solved. As $T\to\infty$\,, the performances of the
 portfolios approach the optimal value. Understandably, the
 portfolios are referred to in Hata,  Nagai,  and Sheu \cite{Hat10} as
 nearly optimal. The other portfolio is ''stationary'' in the sense
 that it is dependent on the value of the economic factor only and 
 is updated ''in real time''. The underperformance probability
 delivered by that investment strategy
 approaches the optimal value as time goes to infinity. In order for
this other portfolio to be asymptotically optimal more restrictions
have to be placed on the model. The  proof of the optimality of the latter
portfolio in Hata,  Nagai,  and Sheu \cite{Hat10} is omitted. The
proof in  Nagai \cite{Nag12} seems to have a gap, as explained below.

The methods of those papers use duality considerations
and rely on connection with risk sensitive control.
A Hamilton--Jacobi--Bellman equation on a finite time horizon is 
analyzed first in order to find an optimal control,
 and, afterwards,   the length of time is allowed to tend to infinity.
In this paper, we approach the problem
 from a different angle. We study  the model tackled in Nagai
 \cite{Nag12} supplemented with a general benchmark.
(Interestingly enough, the presence of a volatile benchmark lends regularity.) 
By a change of variables,  
the setup is cast as a large deviation problem for 
coupled diffusions with time scale separation.
The economic factor can be assumed to 
''live in fast time'' whereas
 the portfolio price is associated with a process that
 ''lives in slow time''. This insight enables
 us to take advantage of the methods developed
 for such diffusions 
 in Liptser \cite{Lip96} and 
Puhalskii \cite{Puh16}. In particular, 
 the empirical measure of the factor process plays a
pivotal role in our study.
Another novel technical feature  is an extensive use of the
saddle--point theory. 

In a fairly general situation, we obtain an asymptotic lower 
 bound on the scaled by the length of the time period logarithmic probability
 of underperformance. Under additional conditions, 
the bound is shown to be tight in the sense that
 there exist stationary  portfolios that approach the lower bound over
 time. Those portfolios generalize  the stationary
 portfolios in Hata, Nagai and Sheu \cite{Hat10} and
 in Nagai \cite{Nag12}. 
If the  assumptions are relaxed,   ''$\epsilon$--optimal''
portfolios are  available whose performance over time falls short of the optimal
value by an arbitrarily small amount so that another limit needs to be taken.
In a standard fashion,
 one can turn two consecutive limits into one so that an
asymptotically optimal portfolio is obtained too.
We are able to
dispose of a  number of assumptions
  in
 Nagai \cite{Nag12} some of  which  are questionable from the modelling
 perspective, e.g., the requirement that
  the  sum of the squared
 differences of the risk--free interest rate and the security
mean
 return rates be bounded below by a quadratic function of the economic
 factor (see the discussion following Remark \ref{re:nonvolat} for more detail), which condition is characterized as being ''crucial'' in Nagai
 \cite{Nag12}. 
There is another important distinction with the results of Nagai
\cite{Nag12} and Hata,  Nagai,  and Sheu
\cite{Hat10}. Both papers require certain stability conditions which
  involve the coefficients of both  the equations for
the economic factor and the equations for the securities. At the same time,
 the model is set in such a way that the economic factor is not
influenced  by the security prices, so, the stability 
conditions 
are arguably at odds with the model's logic.
We use a different stability condition
which is along similar lines as the one in Fleming and Sheu
\cite{FleShe02} and 
concerns  the properties of the economic factor only.
On the technical side, our proofs appear to be less involved than the
ones in Nagai \cite{Nag12} which  could explain why we are able to
tackle a more general model,  we also allow a non--deterministic initial condition for
the economic factor, whereas it is kept fixed in Hata, Nagai, and Sheu
\cite{Hat10} and in Nagai \cite{Nag12}.

This is how this paper is organized. In Section \ref{sec:model}, the
model  is defined,  the choice of an optimal
portfolio is explained intuitively,
main results are stated, and relation to earlier contributions is
discussed in more detail. Section \ref{sec:prelim}  contains auxiliary
results needed for the proofs and the main results are proved in
Section \ref{sec:proof-bounds}.
\section{A model description and  main results}
\label{sec:model}

We consider a portfolio consisting of $n$ risky securities priced
$S^1_t,\ldots,S^n_t$ at time $t$ and a safe security of price $S^0_t$\,.
We assume that the security prices follow the equations
 \begin{equation}
   \label{eq:83}
     \dfrac{dS^i_t}{S^i_t}= a^i(X_t)\,dt+{b^i(X_t)}^T\,dW_t\,,
\end{equation}
 for $i=1,2,\ldots,n$\,,
and  
\begin{equation*}
\frac{dS^0_t}{S^0_t}=r(X_t)\,dt\,, 
\end{equation*}
where $S^i_0>0$ and
 $X_t$ represents an economic factor.
It is governed by the  equation 
\begin{equation}
  \label{eq:14}
  dX_t=\theta(X_t)\,dt+\sigma(X_t)\,dW_t\,.
\end{equation}
In these equations, the
$a^i(x)$ and $r(x)$ are real--valued functions, the $b^i(x)$ are 
$\R^k$--valued functions,
 $\theta(x)$ is  an $\R^l$--valued function, and
$\sigma(x)$ is an 
$l\times k$--matrix--valued function, all being defined for $x\in\R^l$
and ${}^T$ being used to denote the transpose of a matrix or a
vector\,, and
$W_t$ is a $k$--dimensional standard Wiener process.
Accordingly, the process $X=(X_t\,, t\in\R_+)$ is $l$--dimensional. 
As for the initial condition, 
we will assume that
\begin{equation}
  \label{eq:37}
\mathbf Ee^{\gamma\abs{X_0}^2}<\infty\,,
\end{equation}
 for some $\gamma>0$\,.

Benchmark $Y=(Y_t\,,t\in\R_+)$ follows a  similar equation to
those for the risky securities:
\begin{equation*}
  \dfrac{dY_t}{Y_t}=\alpha(X_t)\,dt+\beta(X_t)^T\,dW_t,
\end{equation*}
with
 $\alpha(x)$ being an $\R$--valued function, $\beta(x)$ being an
$\R^k$--valued function, and $Y_0>0$\,.

 All processes
 are defined on a complete probability space
 $(\Omega,\mathcal{F},\mathbf{P})$\,. 
It is assumed, furthermore, 
that the processes $S^i=(S^i_t\,,t\in\R_+)$\,, $X$\,, and
$Y$ are adapted 
 to 
  filtration $\mathbf{F}=(\mathcal{F}_t\,,t\in\R_+)$ and that
 $W=(W_t\,,t\in\R_+)$ is an   $\mathbf{F}$--Wiener process.
 We note that assuming that the $S^i_t$\,, $X_t$\,, and $Y_t$ are
 driven by
the same Wiener process   is not a loss of generality.
(To illustrate the latter point, suppose that we have one risky
security and one economic factor which are driven by
possibly correlated Wiener
processes, i.e.,
\begin{eqnarray*}
\dfrac{dS_{t}}{S_{t}}= \tilde a(X_t)\,dt+\tilde b(X_t)\,dW_{1,t}  &
\text{ and }&
  dX_t=\tilde\theta(X_t)\,dt+\tilde\sigma(X_t)\,dW_{2,t}\,,
\end{eqnarray*}
where $W_{1,t}$ and $W_{2,t}$ are one--dimensional standard Wiener processes such
that $\mathbf E W_{1,t}W_{2,t}=\rho t$ and 
$\tilde a(x)$\,, $\tilde b(x)$\,, $\tilde \theta(x)$\,, 
and $\tilde\sigma(x)$ are
scalar functions. 
This setup can be cast as (\ref{eq:83}) and (\ref{eq:14}) with
$S^1_t=S_t$\,, $W_t=(W_{1,t},W_{3,t})^T$\,,
$b^1(x)=( \tilde b(x),0)$\,,
and $\sigma(x)=(\rho \, \tilde\sigma(x),
\sqrt{1-\rho^2}\, \tilde\sigma(x))$\,, where 
$W_{3,t}$ represents a  one--dimensional standard Wiener process which
is independent of $W_{1,t}$\,.)

We introduce more 
notation and 
assumptions.
We let $a(x)$ denote the $n$--vector with entries  $a^1(x),\ldots,a^n(x)$,
let $b(x)$ denote the $n\times k$ matrix with rows
${b^1(x)}^T,\ldots,{b^n(x)}^T$ and let
$\mathbf{1}$ denote the $n$--vector with unit entries.  
 The matrices $b(x)b(x)^T$ and $\sigma(x)\sigma(x)^T$ are assumed to be
uniformly positive definite and bounded.
 The functions 
${a(x)}$\,,    ${r(x)}$\,,  ${\theta(x)}$\,,
  $\alpha(x)$\,,  $b(x)$\,, $\sigma(x)$\,,     and $\beta(x)$
are assumed to be continuously differentiable with bounded derivatives
 and the function
 $\sigma(x)\sigma(x)^T$ is assumed to be twice continuously differentiable. 
The function  $\abs{\beta(x)}^2$  is assumed to be  bounded and
 bounded
away from zero. 
(We will also indicate how the results change if the benchmark ''is not
volatile'' meaning  that $\beta(x)=0$\,.)
The  following ''linear growth'' condition is assumed:
for some $K>0$ and all $x\in \R^l$\,,
\begin{equation}
  \label{eq:42}
\abs{a(x)}+\abs{r(x)}+\abs{\alpha(x)}+
    \abs{\theta(x)}\le K(1+\abs{x})\,.
\end{equation}
Under those hypotheses, the processes $S^i$\,, $X$\,, and $Y$ are well
defined, see, e.g., Chapter 5 in Karatzas and Shreve \cite{KarShr88}.

The investor holds $l^i_t$ shares of risky security $i$ and $l^0_t$ shares
 of the safe security at time $t$\,,
 so the total wealth
  is given by
$Z_t=\sum_{i=1}^nl^i_tS^i_t+l^0_tS^0_t$\,.
Portfolio
$  \pi_t=(\pi^1_t,\ldots,\pi^n_t)^T$
specifies the proportions of the total wealth invested in the risky
securities so that, for $i=1,2,\ldots,n$,
$l^i_tS^i_t=\pi^i_t Z_t$\,.
The processes $\pi^i=(\pi^i_t\,,t\in\R_+)$ are assumed to be
$\mathcal{B}\otimes\mathcal{F}_t$--progressively measurable, where $\mathcal{B}$
denotes the Borel $\sigma$--algebra on $\R_+$,
 and such that
$\int_0^t{\pi^i_s}^2\,ds<\infty\, , t\in\R_+\,,$ a.s. We do not impose 
 any other restrictions on the magnitudes of the $\pi^i_t$ so that
unlimited borrowing and shortselling are allowed.
Let
\begin{equation*}
    L^\pi_t=\frac{1}{t}\,\ln\bl(\frac{Z_t}{Y_t}\br)\,.
\end{equation*}
Given $q\in\R$\,, 
the objective is to minimize $\liminf_{t\to\infty}
(1/t)\ln\mathbf{P}(L^\pi_t\le q)$  over all portfolios
$\pi=(\pi_t\,,t\in\R_+)$ and identify
portfolios for which the bound is attained.

Since
the amount of wealth invested in the safe
security is $(1-\sum_{i=1}^n \pi^i_t)Z_t$\,, in a standard fashion
by using the self financing condition, cf. Nagai \cite{Nag12},
 we obtain that
\begin{equation*}
  \dfrac{dZ_t}{Z_t}=\sum_{i=1}^n\pi^i_t\,\dfrac{dS^i_t}{S^i_t}+
\bl(1-\sum_{i=1}^n\pi^i_t\br)\,\dfrac{dS^0_t}{S^0_t}\,.
\end{equation*}

Assuming that $Z_0=Y_0$ and letting $c(x)=b(x)b(x)^T$\,, 
 we have
by It\^o's lemma that, cf.  Nagai \cite{Nag12} and Pham \cite{Pha03},
\begin{multline}
\label{eq:1}
  L^\pi_t=
\frac{1}{t}\,\int_0^t
\bl(\pi_s^T a(X_s)+(1-\pi_s^T \ind)r(X_s)
-\frac{1}{2}\,\pi_s^T c(X_s)\pi_s-\alpha(X_s)
+\frac{1}{2}\,\abs{\beta(X_s)}^2\br)\,ds\\+
\frac{1}{t}\,\int_0^t\bl(b(X_s)^T\pi_s -\beta(X_s)\br)^T
\,dW_s\,.
\end{multline}
The following piece of notation comes in useful.
Given  $z\in\R^d$ and
positive definite  symmetric
 $d\times
d$--matrix $V$\,, we denote
 $\norm{z}^2_{V}=z^TVz$\,.
Let, for $u\in\R^n$ and $x\in\R^l$\,,
\begin{subequations}
  \begin{align}
      \label{eq:4}
  M(u,x)&=u^T (a(x)- r(x)\ind )
-\frac{1}{2}\,\norm{u}^2_{c(x)}+r(x)-\alpha(x)
+\frac{1}{2}\,\abs{\beta(x)}^2\intertext{ and }
  \label{eq:8}
N(u,x)&= b(x)^Tu-\beta(x)\,.
\end{align}
\end{subequations}

A  change of variables   brings equation  (\ref{eq:1}) to the form
 \begin{equation}
  \label{eq:5}
  L^\pi_t=
\int_0^1 M(\pi_{ts},X_{ts})\,ds
+\frac{1}{\sqrt{t}}\,\int_0^1 N(\pi_{ts},X_{ts})^T\,dW_{s}^t\,,
\end{equation}
where 
$W^t_s=W_{ts}/\sqrt{t}$\,.
We note that $W^t=(W^t_s,\,s\in[0,1])$ is a Wiener process relative to
$\mathbf{F}^t=(\mathcal{F}_{ts},\,s\in[0,1])$\,. The righthand side of
\eqref{eq:5}  can
be viewed as a
diffusion process with a small diffusion coefficient which
''lives in slow time'' represented by the variable $s$\,, 
whereas in $X$ and $\pi$ ''time'' is
 accelerated by a factor of $t$\,.
Similar setups have been considered in Liptser \cite{Lip96} and 
Puhalskii \cite{Puh16}. Those papers show that in order to study the
large deviation properties of the ''slow'' process it is convenient to 
work with the pair that comprises the slow process and the empirical
measure of the fast process. For  equations (\ref{eq:14}) and
\eqref{eq:5},
this means working with  the pair $(L^\pi_t,\mu_t)$\,,
where  $\mu_t=(\mu_t(ds,dx))$ represents the empirical process of
$(X_{ts}\,, s\in[0,1])$\,, which is
defined by the relation
\begin{equation*}
  \mu_t([0,s],\Gamma)=\int_0^s\chi_{\Gamma}(X_{t\tilde s})\,d\tilde s\,,
\end{equation*}
with $\Gamma$ representing an arbitrary Borel subset of $\R^l$
and with
 $\chi_{\Gamma}(x)$ 
representing the indicator function of the set $\Gamma$\,.
Letting  $\pi^t_s=\pi_{ts}$ and
$X^t_s=X_{ts}$ in (\ref{eq:5}) obtains that
\begin{multline}
  \label{eq:5a}
  L^\pi_t=
\int_0^1 M(\pi_{s}^t,X^t_s)\,ds
+\frac{1}{\sqrt{t}}\,\int_0^1 N(\pi_{s}^t,X_s^t)^T\,dW_{s}^t\\
=\int_0^1\int_{\R^l} M(\pi_{s}^t,x)\,\mu_t(ds,dx)
+\frac{1}{\sqrt{t}}\,\int_0^1 N(\pi_{s}^t,X_s^t)^T\,dW_{s}^t\,.
\end{multline}
We note that both $X^t=(X^t_s,s\in[0,1])$ and $\pi^t=(\pi^t_s,s\in[0,1])$ are
$\mathbf{F}^t$--adapted.

Since, by \eqref{eq:14} and It\^o's lemma, for 
twice continuously differentiable 
function $f$ on $\R^l$\,, with $\nabla f$ and $\nabla^2 f$ 
denoting the gradient and the Hessian of $f$\,, respectively,
and with  $\text{tr}$ standing for the trace of a square matrix,
\begin{multline}
  \label{eq:60}
    f(X_t)=f(X_0)+\int_0^t \nabla f(X_s)^T\theta(X_s)\,ds+
\frac{1}{2}\,\int_0^t \text{tr}\,\bl(\sigma(X_s)\sigma(X_s)^T\nabla^2
f(X_s)\br) \,ds\\
+\int_0^t \nabla f(X_s)^T\sigma(X_s)\,dW_s
\end{multline}
and since the process 
\begin{multline*}
  \bl(\exp\bl(\int_0^t (-\lambda 
N(\pi_s,X_s)+\sigma(X_s)^T\nabla f(X_s))^T\,dW_s\\
-\frac{1}{2}\,\int_0^t \abs{-\lambda 
N(\pi_s,X_s)+\sigma(X_s)^T\nabla f(X_s)}^2\,ds\br)\,,t\in\R_+\br)
\end{multline*}
 is a local martingale relative to $\mathbf F$\,,
 by  (\ref{eq:5a}) and \eqref{eq:60}, for $\lambda\ge0$\,,
\begin{multline}
  \label{eq:1a}
    \mathbf{E}\exp\Bl(-t\lambda L^{\pi}_t-
t\int_0^1 -\lambda M(\pi^t_s,X^t_s)\,ds+f(X_t)-f(X_0)\\-
t\int_0^1\nabla f(X^t_s)^T\,\theta(X^t_s)\,ds
-\frac{t}{2}\,\int_0^1\text{tr}\,({\sigma(X^t_s)}{\sigma(X^t_s)}^T\,
\nabla^2f(X^t_s))\,ds
\\-\frac{t}{2}\,\int_0^1\abs{-\lambda 
N(\pi^t_s,X^t_s)+\sigma(X^t_s)^T\nabla f(X^t_s)}^2\,ds\Br)\le 1\,.
\end{multline}
Intuitively, if we assume that equality prevails in \eqref{eq:1a}, 
which would be the case under certain integrability conditions on
$\pi_s$ and $\nabla f(X_s)$\,, then
 $L^\pi_t$ is ''maximized'' 
by minimizing  the integrals over $\pi^t_s$, i.e.,
by choosing 
$\pi^t_s=u(X^t_s)$ with $u(x)$ attaining 
$\inf_{u\in\R^n}\bl( -\lambda M(u,x)+
\abs{-\lambda 
N(u,x)+\sigma(x)^T\nabla f(x)}^2/2\bl)$\,. For that portfolio,
\begin{multline*}
      \mathbf{E}\chi_{\{L^{\pi}_t\le q\}}\exp\Bl(
-t\lambda q+f(X_t)-f(X_0)\\-
t\int_0^1\int_{\R^l}\bl(\inf_{u\in\R^n}\bl( -\lambda M(u,x)+
\frac{1}{2}\,\abs{-\lambda 
N(u,x)+\sigma(x)^T\nabla f(x)}^2\br)\\
+\nabla f(x)^T\,\theta(x)+
\frac{1}{2}\,\text{tr}\,({\sigma(x)}{\sigma(x)}^T\,
\nabla^2f(x)\br)\mu_t(ds,dx)
\Br)\le 1\,.
\end{multline*}
Consequently,
\begin{multline}
  \label{eq:26}
  \frac{1}{t}\,\ln      \mathbf{E}\chi_{\{L^{\pi}_t\le q\}}e^{f(X_t)-f(X_0)}
\le \lambda q+
\sup_{\nu\in\mathcal{P}}\int_{\R^l}\bl(\inf_{u\in\R^n}\bl( -\lambda M(u,x)\\+
\frac{1}{2}\,\abs{-\lambda 
N(u,x)+\sigma(x)^T\nabla f(x)}^2\br)
+\nabla f(x)^T\,\theta(x)\\+
\frac{1}{2}\,\text{tr}\,({\sigma(x)}{\sigma(x)}^T\,
\nabla^2f(x))\br)\nu(dx)\,.
\end{multline}
The ''best'' upper bound on the normalized logarithmic shortfall probability
$(1/t)\ln      \mathbf{P}(L^{\pi}_t\le q)$ is
obtained by minimizing  the righthand side over
$\lambda\ge0$ and a suitable collection
of functions $f$ so that  an optimal portfolio should be
apparently associated with
$\lambda$ and $f$ that
minimize the righthand side of \eqref{eq:26}. The main results of the paper bear out
that intuition. Furthermore, we show that the upper bound for that
particular portfolio choice furnishes a lower bound for all portfolios.

Before we state the main results, more conditions are in order.
We assume the stability condition that
there exist  bounded Borel--measurable 
function $\Phi(x)$ with values in the set of $l\times n$--matrices,
where $x\in\R^l$\,,
and  positive definite symmetric $l\times
l$--matrix $\Psi$ such that
\begin{equation}
  \label{eq:45}
  \limsup_{\abs{x}\to\infty}\,
\bl(\theta(x)-\Phi(x)(a(x)-r(x)\mathbf1
)\br)^T\,\frac{\Psi x}{\abs{x}^2}<0\,.
\end{equation}
If $\Phi(x)=0$ and $\Psi$ is the $l\times l$-identity matrix, then
one 
recovers Has'minskii's drift condition.

The following nondegeneracy
condition is also needed. Let $I_k$ denote the $k\times
  k$--identity matrix and let $Q_1(x)=I_k-b(x)^Tc(x)^{-1}b(x)$\,.
The matrix $Q_1(x)$ represents the orthogonal 
projection operator onto the null space      of
$b(x)$ in $\R^k$\,.
 We will assume that
\begin{itemize}
\item[(N)]
  \begin{enumerate}
  \item 
 The matrix $\sigma(x)Q_1(x)\sigma(x)^T$ is uniformly
  positive definite.
\item The quantity
  $
  \beta(x)^TQ_2(x)\beta(x)
  $ is bounded away from zero,
where\\
\begin{equation}
  \label{eq:81}
Q_2(x)=Q_1(x)
\bl(I_k-\sigma(x)^T(\sigma(x)Q_1(x)\sigma(x)^T)^{-1}\sigma(x)\br)
Q_1(x)\,.
\end{equation}

  \end{enumerate}
\end{itemize}
This condition admits the following geometric interpretation.
The matrix $\sigma(x)Q_1(x)\sigma(x)^T$ 
is uniformly positive definite 
 if and only if
the ranges of $\sigma(x)^T$ and  $b(x)^T$ do not have common
nontrivial subspaces and, in addition, arbitrary nonzero
vectors from those respective ranges are
 at angles bounded away from zero,
if and only if 
the matrix
$c(x)-b(x)\sigma(x)^T(\sigma(x)\sigma(x)^T)^{-1}\sigma(x)  b(x)^T$
is uniformly positive definite.
Also,      $\beta(x)^TQ_2(x)\beta(x)$  is bounded away from zero
if and only if the projection of $\beta(x)$ onto the null space
 of $b(x)$ is of  length bounded away from zero 
and is at angles bounded away from zero to all nonzero vectors from
the projection of the range of $\sigma(x)^T$ onto that null space.
Under part 1 of condition (N), we have that $k\ge n+l$
and the rows of the matrices $\sigma(x)$ and $b(x)$ are linearly
independent. 
Part 2 of condition (N) implies that $\beta(x)$ does not belong to the
sum of the ranges of $\sigma(x)^T$ and $b(x)^T$   so that $k>n+l$\,.
Part 1 of condition (N) is essential for the developments in this paper
while part 2 may be disposed of at the expense of certain additional
assumptions.

Let
 $\mathcal{P}$ represent the set of probability measures $\nu$ on $\R^l$
such that $\int_{\R^l}\abs{x}^2\,\nu(dx)<\infty$\,.
Let  $\mathbb{P}$ represent the set of probability densities 
$m=(m(x)\,,x\in\R^l)$
 such that $\int_{\R^l}\abs{x}^2\,m(x)\,dx<\infty$\,.
 Let
  $\mathbb{C}^2$,  $\mathbb{C}^2_0$,
 $\mathbb{C}^1$, and $\mathbb{C}^1_\ell$  represent the set of real--valued 
twice continuously differentiable
 functions on $\R^l$,
the set of real--valued 
compactly supported 
twice continuously differentiable
 functions on $\R^l$\,, 
the set of real--valued 
 continuously differentiable
 functions on $\R^l$,
and the set of real--valued 
 continuously differentiable
 functions on $\R^l$ whose gradients satisfy the linear growth
 condition,  respectively.
For  $\mathbb C_0^2$--function $f$\,,
density $m\in\mathbb
  P$ and $\lambda\ge0$\,, we let
\begin{multline}
  \label{eq:62}
  G(\lambda,f,m)=\int_{\R^l}\Bl( 
-\lambda\sup_{u\in\R^n}\bl(  M(u,x)
- \frac{1}{2}\,\lambda
\abs{N(u,x)}^2+ \nabla f(x)^T\sigma(x) N(u,x)\br)\\+
\nabla f(x)^T\,\theta(x)+
\frac{1}{2}\,\abs{{\sigma(x)}^T\nabla f(x)}^2
+\frac{1}{2}\, \text{tr}\,\bl({\sigma(x)}{\sigma(x)}^T\nabla^2 f(x)\br)\,
\Br)\,m(x)\,dx\,.
\end{multline}
Let us define
\begin{equation}
  \label{eq:29}
  F(\lambda)=
 \sup_{m\in \mathbb
  P}\inf_{f\in\mathbb{C}_0^2}G(\lambda,f,m)\,.
\end{equation}
One can see that  $F(0)\le0$ and, moreover, $F(0)=0$ provided condition
\eqref{eq:45} holds with $\Phi(x)=0$\,. 
Let  \begin{align}
  \label{eq:36}
   J^{\text{s}}_q=&  \sup_{\lambda\ge0}(-\lambda q-F(\lambda))\,,
\end{align}
the latter quantity being nonnegative by $F(0)$ being nonpositive and
superscript ''s'' standing for ''shortfall''.
  We  show in Lemma \ref{le:conc} and Lemma \ref{le:minmax} below that 
the
 function $F(\lambda)$ is strictly  convex,
is continuously
 differentiable for $\lambda>0$\,, and converges to $\infty$  superlinearly 
as $\lambda\to\infty$\,, so  
the supremum in  \eqref{eq:36}
 is attained, at  unique $\hat\lambda$\,.
Furthermore, 
by Lemma \ref{le:saddle_3} below, if either $\hat\lambda>0$ or condition
\eqref{eq:45} holds with $\Phi(x)=0$\,, then the function $
\lambda q+ G(\lambda, f,m)$\,,
being  convex  in $(\lambda,f)$ and 
concave  in $m$\,, has      saddle
point
$((\hat\lambda,\hat f),\hat m)$ in $\R_+\times 
(\mathbb C^1_\ell\cap \mathbb C^2)\times
\mathbb P$\,, with 
$\hat\lambda$\,, $\nabla \hat f$\,, and $\hat m$ being specified
uniquely.
In addition, the following equations are satisfied:
\begin{subequations}
  \begin{multline}
  \label{eq:103'}
-\hat\lambda\bl( M(\hat u(x),x)
-\frac{1}{2}\,\hat\lambda
\abs{N(\hat u(x),x)}^2+\nabla\hat f(x)^T\sigma(x) N(\hat u(x),x)\br)
+\nabla  \hat f(x)^T\theta(x)\\
+\frac{1}{2}\,\abs{\sigma(x)^T\nabla\hat  f(x)}^2
+\frac{1}{2}\, \text{tr}\,(\sigma(x)\sigma(x)^T
\,\nabla^2  \hat f(x))=F(\hat\lambda)
\end{multline}
and \begin{multline}
  \label{eq:104'}
                 \int_{\R^l}
\bl(\nabla  h(x)^T(-\hat\lambda \sigma(x) N(\hat u(x),x)+\theta(x)
+\sigma(x)\sigma(x)^T\nabla\hat  f(x))
\\+\frac{1}{2}\, \text{tr}\,(\sigma(x)\sigma(x)^T
\,\nabla^2  h(x))
\br)\, \hat m(x)\,dx=0\,,
\end{multline}
\end{subequations}
with $\hat u(x)$ being given by the $u$ that attains supremum in
\eqref{eq:62} for $\lambda=\hat\lambda$ and $f=\hat f$
so that
\begin{equation}
  \label{eq:69}
        \hat u(x)=\frac{1}{1+\hat\lambda}\,c(x)^{-1}\bl(a(x)-r(x)\mathbf1
+\hat\lambda b(x)\beta(x)+b(x)\sigma(x)^T\nabla\hat
  f(x)\br)
\end{equation}
   and with \eqref{eq:103'} and \eqref{eq:104'} holding
for all $x\in\R^l$ and for all $h\in\mathbb C_0^2$\,, respectively.
Also, $\hat m$ can be assumed to be bounded, positive 
 and continuously differentiable.
In effect, \eqref{eq:103'} and \eqref{eq:104'} represent
Euler--Lagrange equations for $G(\hat\lambda,f,m)$
at $(\hat f,\hat m)$\,.
Equation \eqref{eq:103'} is known as an ergodic Bellman equation,
see Fleming and Sheu \cite{FleShe02}, Kaise and Sheu \cite{KaiShe06},
Hata, Nagai, and Sheu \cite{Hat10}, and 
equation \eqref{eq:104'} signifies that $\hat m$ is the
invariant density of a certain diffusion.
Kaise and Sheu \cite{KaiShe06}, see also Ichihara \cite{Ich11},
 develop an elegant theory of the
ergodic Bellman equation which is essential for our study.
One hopes that 
 the portfolio $\hat\pi=(\hat \pi_t\,,t\in\R_+)$
such that $\hat\pi_t=\hat u(X_t)$ is asymptotically optimal.

\begin{theorem}
  \label{the:bounds}
Let us suppose that either  $\hat\lambda>0$ or condition
\eqref{eq:45} holds with $\Phi(x)=0$\,. Then,
for arbitrary portfolio $\pi=(\pi_t\,,t\in\R_+)$\,,
\begin{equation*}
    \liminf_{t\to\infty}\frac{1}{t}\ln
\mathbf{P}(L^\pi_t< q)\ge -J^{\text{s}}_q\,.
\end{equation*}
\end{theorem}
\begin{remark}
Let  $\mathbb C_b^2$ represent the set of twice continuously
differentiable functions on $\R^l$ with bounded second derivatives.
It is shown in Lemma \ref{le:minmax} below
 that  $F(\lambda)=
 \inf_{f\in\mathbb{C}_b^2}\sup_{m\in \mathbb
  P}G(\lambda,f,m)$\,.
 Thus, the assertion of Theorem \ref{the:bounds}
 is consistent with the intuition
 provided after \eqref{eq:26}.
\end{remark}
Let, given $\lambda\in\R$\,, 
$f\in \mathbb C^2$\,, 
 and $v=(v(x),\,x\in\R^l)$\,,
  \begin{multline}
    \label{eq:61}
      \breve H(x;\lambda, f,v)= 
-\lambda  M(v(x),x)
+\frac{1}{2}\,\abs{-\lambda
N(v(x),x)+{\sigma(x)}^T\nabla f(x)}^2
\\+\nabla f(x)^T\,\theta(x)
+\frac{1}{2}\, \text{tr}\,\bl({\sigma(x)}{\sigma(x)}^T\nabla^2 f(x)\br)\,.
\end{multline}
By \eqref{eq:103'}, for all $x\in\R^l$\,,
\begin{equation}
  \label{eq:19}
  \breve H(x;\hat\lambda,\hat f,\hat u)=F(\hat\lambda)\,.
\end{equation}
In addition,
  by  Remark \ref{re:diffusion} below,
  $F(\hat\lambda)=\inf_{f\in\mathbb C^2}
\sup_{x\in\R^l}\breve H(x;\hat\lambda,f,\hat u)$\,.

Given $\tau>0$\,, let $\hat u^\tau(x)=\hat u(x)\chi_{[0,\tau]}(\abs{x})$\,.
Let us introduce the following condition
\begin{equation}
 \label{eq:22}
  \limsup_{\tau\to\infty}
\inf_{f\in\mathbb C^2_b}\sup_{x\in\R^l}
\breve H(x;\hat\lambda,f,\hat u^\tau)
\le F(\hat\lambda)\,.
\end{equation}
By Lemma \ref{le:condition} below, \eqref{eq:22} holds provided that
 either $\hat\lambda=0$ or
there exist $\varrho>0$\,, $C_1>0$ and $C_2>0$ such that, for all $x\in\R^l$\,,
\begin{equation}
  \label{eq:31} (1+\varrho)   \norm{b(x)\sigma(x)^T\nabla\hat
  f(x)}^2_{c(x)^{-1}}
-\norm{a(x)-r(x)\mathbf1}^2
_{c(x)^{-1}}\le C_1\abs{x}+C_2\,.
\end{equation}
We also introduce  the following stronger version of \eqref{eq:31}:
\begin{equation}
  \label{eq:39}
\lim_{\abs{x}\to\infty}\bl(
  (1+\varrho)   \norm{b(x)\sigma(x)^T\nabla\hat
  f(x)}^2_{c(x)^{-1}}
-\norm{a(x)-r(x)\mathbf1}^2
_{c(x)^{-1}}\br)=-\infty\,.
\end{equation}
 Let $\hat\pi=(\hat u(X_t)\,,t\in\R_+)$ and
 $\hat\pi^\tau=(\hat u^\tau(X_t)\,,t\in\R_+)$\,.
\begin{theorem}
  \label{the:risk-sens}
Suppose that  
\eqref{eq:45} holds with $\Phi(x)=0$\,.
\begin{enumerate}
\item 
If \eqref{eq:22} holds, 
then 
 \begin{align*}
\lim_{\tau\to\infty}\liminf_{t\to\infty}\frac{1}{t}\ln
\mathbf{P}(L^{\hat\pi^\tau}_t< q)=
\lim_{\tau\to\infty}\limsup_{t\to\infty}\frac{1}{t}\ln
\mathbf{P}(L^{\hat\pi^\tau}_t\le q)= -J^{\text{s}}_{q}\,.
\end{align*}
\item
If, in addition,  \eqref{eq:39} holds,
then
  \begin{equation*}
\lim_{t\to\infty}\frac{1}{t}\ln
\mathbf{P}(L^{\hat\pi}_t< q)=
\lim_{t\to\infty}\frac{1}{t}\ln
\mathbf{P}(L^{\hat\pi}_t\le q)= -J^{\text{s}}_{q}\,.
  \end{equation*}
\end{enumerate}
\end{theorem}
\begin{remark}
Conditions \eqref{eq:31} and \eqref{eq:39} are modelled on respective conditions
 (2.25)  in Nagai \cite{Nag12} and   (4.1) in Nagai \cite{Nag03}.
As the proof of Lemma \ref{le:condition}  shows, an upper bound on  the righthand side of \eqref{eq:31} can be
  allowed to grow at a subquadratic rate.
\end{remark}
\begin{remark}
  One can see that
  \begin{equation*}
    \inf_{f\in\mathbb C^2_b}\sup_{x\in\R^l}
\breve H(x;\hat\lambda,f,\hat u^\tau)
=\inf_{f\in\mathbb C^2}\sup_{x\in\R^l}
\breve H(x;\hat\lambda,f,\hat u^\tau)\,.
  \end{equation*}
\end{remark}
\begin{remark}
  The limit in \eqref{eq:39} holds provided
  \begin{equation*}
    \limsup_{\abs{x}\to\infty}
\frac{1}{\abs{x}^2}\bl(\norm{b(x)\sigma(x)^T\nabla\hat
  f(x)}^2_{c(x)^{-1}}
-\norm{a(x)-r(x)\mathbf1}^2
_{c(x)^{-1}}\br)<0\,.
  \end{equation*}
It would be nice to have a condition expressed in terms of the
coefficients of the model equations but this seems to be an open
problem. \end{remark}
\begin{remark}
    Under the hypotheses of part 1 of Theorem \ref{the:risk-sens},
there exists strictly increasing function $t(\tau)$ such that
$  \lim_{\tau\to\infty}(1/t(\tau))\ln
\mathbf{P}(L^{\hat\pi^\tau}_{t(\tau)}< q)= -J^{\text{s}}_{q}\,.
$
Letting $\tau(t)$ represent the inverse function to $t(\tau)$\,,
we have that
$  \lim_{t\to\infty}(1/t)\ln
\mathbf{P}(L^{\hat\pi^{\tau(t)}}_{t}< q)= -J^{\text{s}}_{q}\,,
$
so, $(\hat u^{\tau(t)}(X_t)\,,t\ge0)$ is an asymptotically optimal
portfolio.
\end{remark}
\begin{remark}
\label{re:nonvolat}When
$\beta(x)=0$\,,
  the proofs of  Theorems \ref{the:bounds} and
 \ref{the:risk-sens}
 go through and their assertions are maintained   provided
part 1 of condition (N) is satisfied
 and $\inf_{x\in\R^l}(r(x)-\alpha(x))<q$\,, with the
$\epsilon$--optimal portfolio being defined similarly.
If $\inf_{x\in\R^l}(r(x)-\alpha(x))\ge q$\,, 
then investing  in the safe security only
is obviously optimal.
\end{remark}
For the case where  $\alpha(x)=r(x)$ 
and $\beta(x)=0$\,, the control in \eqref{eq:69}
appears in Theorem 2.5  in Nagai \cite{Nag12}, which 
obtains the
 limit in part 2 of Theorem \ref{the:risk-sens}.
Instead of condition 
\eqref{eq:39}, it is required  in Nagai \cite{Nag12}
 that $\norm{b(x)\sigma(x)^T\nabla\hat
  f(x)}^2_{c(x)^{-1}}
-\norm{a(x)-r(x)\mathbf1}^2
_{c(x)^{-1}}<0$\,, for all $x$ (see (2.25) in Nagai
\cite{Nag12}). Since it is assumed in Nagai \cite{Nag12} that 
$\norm{a(x)-r(x)\mathbf1}^2_{c(x)^{-1}}$ 
is bounded below by a quadratic function of $\abs{x}$
(see (2.21) there)   and since $\abs{\nabla\hat f(x)}$ is, at most, of linear
growth, that condition implies \eqref{eq:31}. It does not imply \eqref{eq:39}.
As  mentioned in the Introduction,
 we have our doubts as to the proof of Theorem 2.5 in
Nagai \cite{Nag12} being sound: the last display of the proof on p.660
 doesn't seem to be substantiated in that it is not clear how the term 
$\check{ E}\int_0^Te^{- w(X_s)}(-\chi)\,ds$ on the preceding line
 is tackled, $-\chi$ being a positive number,
 e.g., why should $\check{ E}e^{-
  w(X_s)}<\infty$\,, given that $-w(x)$ grows 
no slower than quadratically
with $\abs{x}$ ?
Similar terms were treated more carefully in Kuroda and Nagai \cite{KurNag02} 
and in Nagai \cite{Nag03}, e.g., in Nagai \cite{Nag03} it is required
that $\norm{b(x)\sigma(x)^T\nabla\hat
  f(x)}^2_{c(x)^{-1}}
-\norm{a(x)-r(x)\mathbf1}^2
_{c(x)^{-1}}\to-\infty$\,, as $\abs{x}\to\infty$\,, which condition is
equivalent to \eqref{eq:39} when $\norm{a(x)-r(x)\mathbf1}^2
_{c(x)^{-1}}$ is bounded below by a quadratic function of $\abs{x}$ and
may suffice to conclude the proof in Nagai \cite{Nag12}.
 Theorem 2.4 in 
Nagai \cite{Nag12} does not require the condition
but, as we have mentioned in the Introduction,
 it produces different portfolios for different time horizons $T$\,.   Besides,
additional assumptions are introduced both in Theorem 2.4
and in Theorem 2.5 in Nagai \cite{Nag12}
 (see (2.19) and (2.20) there) along with the requirement that $0<q<-F'(0+)$\,,
the righthand side of the latter inequality 
 ruling out the possibility that $\hat\lambda=0$\,. 
(Interestingly enough, the condition that   $q>0$ is consistent with
Remark \ref{re:nonvolat}.)
Stability condition \eqref{eq:45} is assumed to hold in
 Nagai \cite{Nag12} with
$\Phi(x)=\sigma(x)b(x)^Tc(x)^{-1}$ and $\Psi$ being the $l\times
l$--identity matrix. 
As argued in the Introduction, 
imposing a stability condition on $X_t$ only, which is
what the choice $\Phi(x)=0$ does, is more natural from an application
point of view. 

In the  Gaussian case,  finding the 
 portfolio
$\hat\pi$  reduces to
solving an algebraic Riccati equation.
Let us  assume that
 $a(x)$\,, $r(x)$\,,  
$\alpha(x)$\,, and $\theta(x)$ are affine functions of $x$ and that the diffusion
coefficients are constant. More specifically, let
\begin{subequations}
  \begin{align*}
a(x)=A_1x+a_2\,, \\
r(x)=r_1^Tx+r_2\,,\\
\alpha(x)=\alpha_1^Tx+\alpha_2\,, \\
\theta(x)=\Theta_1 x+\theta_2\,,
\intertext{and}
b(x)=b,\;\beta(x)=\beta,\;\sigma(x)=\sigma\,,
  \end{align*}
\end{subequations}
where $A_1\in \R^{n\times l}$\,, $a_2\in\R^n$\,, 
$r_1\in\R^l$\,, $r_2\in \R$\,, $\alpha_1\in \R^l$\,,
$\alpha_2\in\R$\,, $\Theta_1\in \R^{l\times l}$\,,
 $\theta_1\in \R^l$\,, $b$ is an $n\times k$--matrix such that
the matrix $bb^T$ is positive definite, $\beta$ is a non--zero
$k$--vector,
 and
$\sigma$ is an $l\times k$--matrix such that the matrix $\sigma\sigma^T$ is
positive definite.
Condition \eqref{eq:45}  is fulfilled provided the pair
$(A_1-\mathbf1 r_1^T,\Theta_1)$ 
is detectable, i.e., there exists $l\times n$--matrix $\Phi$ such that
the matrix $\Theta_1-\Phi (A_1-\mathbf1 r_1^T)$ is stable,
for, in that case, there exists  symmetric positive definite
$l\times l$--matrix
$\Psi$ such that $ (\Theta_1-\Phi(A_1-\mathbf1 r_1^T))^T\Psi+
\Psi(\Theta_1-\Phi(A_1-\mathbf1 r_1^T))=-I_l$\,,
see, e.g., p.252 in Bellman \cite{Bel70} or Theorem 8.7.2 on p.270 
in Lancaster \cite{Lan69},
so one can take $\Phi(x)=\Phi $\,, where
$I_l$ stands for the $l\times l$--identity matrix. 
Consequently, \eqref{eq:45} holds with $\Phi=0$ when the
matrix $\Theta_1$ is stable.

Let\begin{align*}
  A&=\Theta_1-\frac{\hat\lambda}{1+\hat\lambda}\,\sigma b^Tc^{-1}(A_1-\mathbf 1
  r_1^T),\\
B&=\sigma\sigma^T-\frac{\hat\lambda}{1+\hat\lambda}\,\sigma
  b^Tc^{-1}b\sigma^T\,,\intertext{and}
C&=
(A_1-\mathbf1 r_1^T)^Tc^{-1}(A_1-\mathbf1 r_1^T)\,.
\end{align*}
Let us  suppose that there exists
 symmetric $l\times l$--matrix $\hat P_1$ 
that satisfies the algebraic Riccati  equation
\begin{equation*}
 A^T\hat P_1+\hat P_1A+ \hat P_1B\hat P_1
    -\frac{\hat\lambda}{1+\hat\lambda}\,C=0\,.
\end{equation*}
  Conditions
for the existence of solutions can be found
in Fleming and
   Sheu \cite{FleShe02}, 
   Willems \cite{MR0308890}, and  Wonham \cite{MR0239161}.
For instance, if
$\Theta_1$ is a  stable matrix, then  the pairs $(A,\sigma)$
and $(A_1-\mathbf 1 r_1^T,A)$ are
stabilizable and detectable, respectively, so,
by Theorem 4.1 in Wonham \cite{MR0239161}
 there  exists a negative semidefinite symmetric matrix that
satisfies the equation and  
  the matrix $D=A+B\hat P_1$
is stable.  Lemma 3.3 in Fleming and Sheu \cite{FleShe02}
 asserts the uniqueness of $\hat P_1$\,, provided that $\Theta_1+\Theta^T_1$ is
 negative definite.
 With $D$ being stable, the equation
\begin{multline*}
D^T
\hat p_2
    -\frac{\hat\lambda}{1+\hat\lambda}\,
(A_1-\mathbf
1r_1^T+b\sigma^T\hat P_1)^Tc^{-1}(a_2-r_2\mathbf1+\hat\lambda
b\beta)\\
-\hat\lambda(r_1-\alpha_1-\hat P_1\sigma\beta)+
\hat P_1\theta_2=0
\end{multline*}
has a unique solution for $\hat p_2$\,. The function 
$\hat f(x)=  x^T\hat P_1x/2+\hat p_2^T x$
solves the ergodic Bellman equation \eqref{eq:103'}, where
\begin{equation}
  \label{eq:9}
        \hat u(x)=\frac{1}{1+\hat\lambda}\,c^{-1}\bl(A_1-\mathbf 1r_1^T
+b\sigma^T\hat P_1)x+
\frac{1}{1+\hat\lambda}\,c^{-1}\bl(a_2-r_2\mathbf1
+\hat\lambda b\beta+b\sigma^T\hat p_2\br)\,.
\end{equation}
If the matrix
$(b\sigma^T\hat P_1)^Tc^{-1}b\sigma^T\hat P_1-
(A_1-\mathbf 1r_1^T)^Tc^{-1}(A_1-\mathbf 1r_1^T)$ is
 negative definite, then   \eqref{eq:39}  holds.
By \eqref{eq:104'},
  $\hat m$ is the invariant density of the linear
diffusion
\begin{multline*}
    dY_t=DY_t\,dt
+\bl(-\frac{\hat\lambda}{1+\hat\lambda}\,\sigma b^Tc^{-1}
(a_2-r_2\mathbf1+\hat\lambda b\beta
+b\sigma^T\hat p_2)+\hat\lambda\sigma\beta
+\sigma\sigma^T\hat p_2+\theta_2\br)\,dt\\+\sigma\,dW_t
\end{multline*}
and 
\begin{multline*}
   F(\hat\lambda)
=- \frac{1}{2}\,\frac{\hat\lambda}{1+\hat\lambda}\,\norm{a_2-r_2
\mathbf1+\hat\lambda b\beta+b\sigma^T\hat p_2}^2_{c^{-1}}
-\hat\lambda(r_2-\alpha_2
+\frac{1}{2}\,\abs{\beta}^2-\beta^T\sigma^T\hat p_2)\\
+\frac{1}{2}\,\hat\lambda^2\abs{\beta}^2
+\frac{1}{2}\,\hat p_2^T\sigma\sigma^T\hat p_2+\hat p_2^T\theta_2+\frac{1}{2}\, \text{tr}\,(\sigma\sigma^T
\,\hat P_1)\,.
\end{multline*}
For the nonbenchmarked case, the  portfolio in \eqref{eq:9}  is obtained
in  Hata, Nagai, and Sheu \cite{Hat10} (see
(2.39) and Theorem 2.2 there). 
 For the optimality of $\hat\pi$\,,  those authors,
who assume that $r_1=0$\,, $\alpha_1=0$\,, $\alpha_2=0$\,, and
$\beta=0$\,,
  in addition to
 requiring that
the matrix $\Theta_1-\sigma b^Tc^{-1}A_1$ be stable
and that the matrix $(b\sigma^T\hat P_1)^Tc^{-1}b\sigma^T\hat P_1-
(A_1-\mathbf 1r_1^T)^Tc^{-1}(A_1-\mathbf 1r_1^T)$ be
negative definite, need that
$(\Theta_1,\sigma)$ be controllable and that $q<-F'(0+)$\,.
Our results relax those restrictions as well as incorporate the case
 of nonzero $r_1$\,, $\alpha_1$\,, $\alpha_2$\,, and $\beta$\,.
It has to be mentioned that the proof of Theorem 2.2 in
  Hata, Nagai, and Sheu \cite{Hat10} is omitted and that the authors
produce also 
non time--homogeneous portfolios that are ''nearly'' optimal under
weaker hypotheses but require the same stability condition. 
\section{Technical preliminaries}
\label{sec:prelim}
In this section, we lay the groundwork for the proofs of the main results.
Let, given $x\in\R^l$\,, $\lambda\ge0$\,, and $p\in\R^l$\,,
\begin{multline}
  \label{eq:80}
\hat   H(x;\lambda,p)= -\lambda\sup_{u\in\R^n}\bl(  M(u,x)
- \frac{1}{2}\,\lambda
\abs{N(u,x)}^2+ p^T\sigma(x) N(u,x)\br)+
p^T\,\theta(x)\\+\frac{1}{2}\,\abs{{\sigma(x)}^Tp}^2\,.
\end{multline}

One can thus write \eqref{eq:62} more compactly as
\begin{equation}
  \label{eq:13}
  G(\lambda,f,m)=\int_{\R^l}\bl(\hat H(x;\lambda,\nabla f(x)) 
+\frac{1}{2}\, \text{tr}\,\bl({\sigma(x)}{\sigma(x)}^T\nabla^2 f(x)\br)\br)
\,m(x)\,dx\,.
\end{equation}
Let \begin{equation*}
  T_\lambda(x)=\sigma(x)\sigma(x)^T-\frac{\lambda}{1+\lambda}\, \sigma(x)b(x)^T
c(x)^{-1}b(x)\sigma(x)^T\,.
\end{equation*}
By part 1 of  condition (N), 
 $T_\lambda(x)$ is a uniformly positive definite
symmetric $l\times
l$--matrix. 
Optimizing on the righthand side of \eqref{eq:80} yields 
\begin{multline}
  \label{eq:3}
  \sup_{u\in\R^n}\bl( M(u,x)
- \frac{1}{2}\,\lambda
\abs{N(u,x)}^2+p^T\sigma(x) N(u,x)\br)
\\=    \frac{1}{2}\,\frac{1}{1+\lambda}\,
\norm{a(x)-r(x)\mathbf1+\lambda
  b(x)\beta(x)+b(x)\sigma(x)^Tp}^2_{c(x)^{-1}}
\\- \frac{1}{2}\,\lambda\abs{\beta(x)}^2+
r(x)-\alpha(x)+\frac{1}{2}\,\abs{\beta(x)}^2-\beta(x)^T\sigma(x)^Tp
  \end{multline}
so that
\begin{multline}
  \label{eq:65}
\hat  H(x;\lambda,p)=
\frac{1}{2}\,p^TT_\lambda(x)p
+\bl(-\frac{\lambda}{1+\lambda}\,(a(x)-r(x)\mathbf1\\
+\lambda
b(x)\beta(x))^Tc(x)^{-1}b(x)\sigma(x)^T+\lambda\beta(x)^T\sigma(x)^T
+\theta(x)^T\br)p\\
-\frac{\lambda}{2(1+\lambda)}\,\norm{a(x)-r(x)\ind+
\lambda b(x)\beta(x)}^2_{c(x)^{-1}}
\\-\lambda(r(x)-\alpha(x)+\frac{1}{2}\,\abs{\beta(x)}^2)
+\frac{1}{2}\,\lambda^2\abs{\beta(x)}^2
\,.
\end{multline}

Drawing on    Bonnans and Shapiro \cite{BonSha00}, we say that, given
topological space  $\mathbb T$\,, 
function $h:\,\mathbb T\to\R$ is 
$\inf$--compact, respectively, $\sup$--compact, if the sets
$\{x\in \mathbb T:\,h(x)\le \delta\}$\,, respectively, 
the sets $\{x\in \mathbb T:\,h(x)\ge \delta\}$\,, are compact for all
$\delta\in\R$\,. (It is worth noting that Aubin and
Ekeland \cite{AubEke84} only require that the sets above be
relatively compact. These two defintions are equivalent 
if the function in question is, in addition, lower
 semicontinuous,  respectively, upper semicontinuous.) 
We endow 
the set $\mathcal{P}$  of probability measures $\nu$ on $\R^l$
such that $\int_{\R^l}\abs{x}^2\,\nu(dx)<\infty$
 with the Kantorovich--Rubinstein distance 
\begin{equation*}
  d_1(\mu,\nu)=\sup\{\abs{\int_{\R^l}g(x)\,\mu(dx)-\int_{\R^l}g(x)\,\nu(dx)}:\;
\frac{\abs{g(x)-g(y)}}{\abs{x-y}}\le 1\text{ for all }x\not=y\}\,.
\end{equation*}
Convergence with respect to $d_1$ is equivalent to weak
convergence coupled with  convergence of the first moments, see, e.g.,
Villani \cite{Vil09}.

We  introduce, for $f\in\mathbb C^2$\,,
\begin{equation}
  \label{eq:59} 
H(x;\lambda,f)= \hat H(x;\lambda,\nabla f(x))
+\frac{1}{2}\, \text{tr}\,\bl({\sigma(x)}{\sigma(x)}^T\nabla^2 f(x)\br)
\end{equation}
so that, for $f\in\mathbb C^2_0$\,, 
\begin{equation}
  \label{eq:88}
    G(\lambda,f,m)=\int_{\R^l} H(x;\lambda, f)\,m(x)\,dx\,.
\end{equation}

For $\kappa>0$\,, we define   $f_\kappa(x)=\kappa\norm{
  x}^2_\Psi/2$ and let
  $\mathcal{A}_\kappa$ denote the convex hull of
$\mathbb C_0^2$ and of the function 
$f_\kappa$\,. The next lemma implies, in particular, that
 $F(\lambda)$ is  finite--valued.
\begin{lemma}
  \label{le:sup-comp}
Suppose that either $\lambda>0$ or that condition \eqref{eq:45} holds
with $\Phi(x)=0$\,. Then, for all $\kappa>0$ small enough, the function
$\int_{\R^l}H(x;\lambda,f_\kappa)\,\nu(dx)$
 is  $\sup$--compact in $\nu\in\mathcal{P}$ 
for the Kantorovich--Rubinstein distance $d_1$\,.
The function 
$\inf_{f\in\mathbb C^2_0}\int_{\R^l}H(x;\lambda, f)\,\nu(dx)$
is $\sup$--compact in $\nu$\,.
Furthermore,  the set
$\bigcup_{\{\lambda:\,\abs{\lambda-\overline\lambda}\le\overline\lambda/2\}}
\{\nu\in \mathcal{P}:\,\inf_{f\in\mathbb C^2_0}\int_{\R^l}H(x;\lambda,
f)\,\nu(dx) \ge \delta\}$ is relatively compact, where 
 $\overline \lambda>0$ and $\delta\in\R$\,.
\end{lemma}
\begin{proof}
By \eqref{eq:65}, 
\begin{multline}
  \label{eq:15}
  \hat  H(x;\lambda,\nabla f_\kappa(x))=
\frac{\kappa^2}{2}\,
\norm{\Psi x}^2_{T_\lambda(x)}
+\kappa\bl(\theta(x)-\Phi(x)(a(x)-r(x)\mathbf1
)\br)^T\Psi x\\
+\kappa\bl(\bl(-\frac{\lambda}{1+\lambda}\,
\sigma(x)b(x)^Tc(x)^{-1}+\Phi(x)\br)\,(a(x)-r(x)\mathbf1
)\br)^T
\Psi x
\\-\frac{\lambda}{2(1+\lambda)}\,
\norm{a(x)-r(x)\ind}^2_{c(x)^{-1}}\\
+\kappa\bl(-\frac{\lambda^2}{1+\lambda}\,
\beta(x)^Tb(x)^Tc(x)^{-1}b(x)\sigma(x)^T+\lambda\beta(x)^T\sigma(x)^T
\br)\Psi x
\\-\frac{\lambda}{2(1+\lambda)}\,
\bl(2\lambda(a(x)-r(x)\ind)^Tc(x)^{-1}b(x)\beta(x)+
\lambda^2 \norm{b(x)\beta(x)}^2_{c(x)^{-1}}\br)\\
-\lambda(r(x)-\alpha(x)+\frac{1}{2}\,\abs{\beta(x)}^2)
+\frac{1}{2}\,\lambda^2\abs{\beta(x)}^2
\,.
\end{multline}
Let us suppose that $\lambda>0$\,. Since $\Phi(x)$ is a bounded function,
by the Cauchy
inequality, there exists $K_1>0$ such that, for all
 $\epsilon>0$\,,
\begin{multline*}
\bl(\bl(-\frac{\lambda}{1+\lambda}\,
\sigma(x)b(x)^Tc(x)^{-1}+\Phi(x)\br)\,(a(x)-r(x)\mathbf1
)\br)^T
 \Psi x\\\le\frac{1}{2\epsilon^2}\,\norm{a(x)-r(x)\mathbf1
}^2_{c(x)^{-1}}+ K_1\, \frac{\epsilon^2}{2 }
\abs{\Psi x}^2\,.
\end{multline*}
By condition \eqref{eq:45}, if $\kappa$ and $\epsilon$ are small enough, then
\begin{multline*}
  \limsup_{\abs{x}\to\infty}
\frac{1}{\abs{x}^2}
\bl(\kappa K_1\, \frac{\epsilon^2}{2 }
\abs{\Psi x}^2+\frac{\kappa^2}{2}\,x^T\Psi
T_\lambda(x)\Psi x\\
+\kappa\bl(\theta(x)-\Phi(x)(a(x)-r(x)\mathbf1
)\br)^T\Psi x\br)<0\,.
\end{multline*}
Finally, given $\epsilon$\,, $\kappa$ can be chosen such that
\begin{equation*}
  \frac{\kappa}{2\epsilon^2}\,\norm{a(x)-r(x)\mathbf1
}^2_{c(x)^{-1}}
-\frac{\lambda}{2(1+\lambda)}\,
\norm{a(x)-r(x)\ind}^2_{c(x)^{-1}}\le0\,.\end{equation*}
Putting everything together and noting that the terms on the lower
two lines of \eqref{eq:15} grow at most linearly with $\abs{x}$\,,
we conclude that, provided $\kappa$ is small enough,
for suitable $K_2$ and $K_3>0$\,, 
\begin{equation}
  \label{eq:23}
 H(x;\lambda,f_\kappa)\le K_2-K_3\abs{x}^2\,.
\end{equation}
By \eqref{eq:45} and \eqref{eq:15},
the latter inequality can also be  fulfilled  if  $\lambda=0$ and $\Phi(x)=0$\,.

Therefore, on introducing $\Gamma_\delta=  \big\{\nu\in\mathcal{P}:\,
 \int_{\R^l} H(x;\lambda,f_\kappa)\,\nu(dx)\ge \delta\big\}
$\,, where $\delta\in\R$\,, we have that
$\sup_{\nu\in\Gamma_\delta}\int_{\R^l}\abs{x}^2\,\nu(dx)<\infty$\,.
 (As a general matter, we assume that $\sup_\emptyset=-\infty$ and
$\inf_\emptyset =\infty$\,.)
In addition, by $H(x;\lambda,f_\kappa)$ being  continuous
in $x$\, and by \eqref{eq:23},
$\int_{\R^l}H(x;\lambda,f_\kappa)\,\nu(dx)$
 is an upper semicontinuous
function of $\nu$\,, so $\Gamma_\delta$ is a closed set.
Thus, by Prohorov's theorem and Lebesgue's dominated convergence theorem,
  $\Gamma_\delta$   is compact.

By \eqref{eq:65} and
\eqref{eq:59}, the function $H(x;\lambda,f)$ is convex in $f$\,. 
Therefore, if $f\in\mathcal{A}_\kappa$\,, then, by
\eqref{eq:65} and  \eqref{eq:23},
$H(x;\lambda,f)$  is
bounded above 
 by an affine function of $x$\,. Since $H(x;\lambda,f)$ is continuous in $x$\,,
  the function $\int_{\R^l}H(x;\lambda, f)\,\nu(dx)$
is upper semicontinuous in $\nu$\,.
Since $f_\kappa\in\mathcal{A}_\kappa$\,,
we obtain that $\inf_{f\in\mathcal{A}_\kappa}\int_{\R^l}H(x;\lambda, f)\,\nu(dx)$
is $\sup$--compact. 
Since $\inf_{f\in\mathcal{A}_\kappa}\int_{\R^l}H(x;\lambda, f)\,\nu(dx)=
\inf_{f\in\mathbb C_0^2}\int_{\R^l}H(x;\lambda, f)\,\nu(dx)$\,,
the latter function is $\sup$--compact.

An examination of the reasoning that led to
 \eqref{eq:23} reveals that   there exist $\overline K_2$ and
$\overline K_3>0$ such that $H(x;\lambda,f_\kappa)\le
\overline K_2-\overline K_3\abs{x}^2$ if $\abs{\lambda-\overline
  \lambda}\le \overline\lambda/2$\,. Therefore,
\begin{multline*}
  \bigcup_{\{\lambda:\,\abs{\lambda-\overline\lambda}\le\overline\lambda/2\}}
\{\nu\in \mathcal{P}:\,\inf_{f\in\mathbb C^2_0}\int_{\R^l}H(x;\lambda,
f)\,\nu(dx) \ge \delta\}\\\subset
  \bigcup_{\{\lambda:\,\abs{\lambda-\overline\lambda}\le\overline\lambda/2\}}
\{\nu\in \mathcal{P}:\,\int_{\R^l}H(x;\lambda,
f_\kappa)\,\nu(dx) \ge \delta\}\\\subset
\{\nu\in \mathcal{P}:\,\overline K_3\int_{\R^l}\abs{x}^2
\,\nu(dx) \le \overline K_2-\delta\}\,.
\qed\end{multline*}
\end{proof}
For $\nu\in\mathcal{P}$\,, 
we let $\mathbb L^{2}(\R^l,\R^l,\nu(dx))$ 
represent the Hilbert space (of the equivalence classes)
of $\R^l$-valued functions $h(x)$ on
$\R^l$ that are square integrable with respect to $\nu(dx)$ equipped
with the norm $\bl(\int_{\R^l}\abs{h(x)}^2\,\nu(dx)\br)^{1/2}$ and we let
 $\mathbb L^{1,2}_0(\R^l,\R^l,\nu(dx))$ 
represent the closure in 
$\mathbb L^{2}(\R^l,\R^l,\nu(dx))$ 
of the set of gradients of $\mathbb
C_0^1$-functions, with $\mathbb{C}^1_0$ denoting
 the set of real-valued 
compactly supported 
 continuously differentiable
 functions on $\R^l$\,. The space $\mathbb
 L^{1,2}_0(\R^l,\R^l,\nu(dx))$ is a Hilbert space too.
 We will use the notation 
$\nabla f$ for the elements of $\mathbb
L^{1,2}_0(\R^l,\R^l,\nu(dx))$\,, although those functions might not
be proper gradients.
 Let $\hat{\mathbb{P}}$ represent the set of probability densities $m$
  such that $m\in\mathbb{P}$\,,
 $m\in\mathbb{W}^{1,1}_{\text{loc}}(\R^l)$\,,
 and  $\sqrt{m}\in\mathbb{W}^{1,2}(\R^l)$\,, where $\mathbb W$
 is used for denoting a Sobolev space, see, e.g., Adams and Fournier
 \cite{MR56:9247}. 
We note that $\hat{\mathbb P}$ is a convex subset of $\mathbb P$\,.
In the next lemma and below,
the 
divergence of a square matrix is defined as the vector whose entries are the
divergencies of the rows of the matrix.
 \begin{lemma}
   \label{le:density_differ}
If, for   $\nu\in\mathcal{P}$\,,
$\inf_{f\in\mathbb C_0^2}
\int_{\R^l} H(x;\lambda, f)\,\nu(dx)>-\infty$\,,
then $\nu$ admits density  which belongs to
$\hat{\mathbb{P}}$\,.
 \end{lemma}
 \begin{proof}
   The reasoning follows that of Puhalskii \cite{Puh16}, cf. Lemma
   6.1, Lemma 6.4, and Theorem 6.1 there.
If there exists $\kappa\in\R$ such that
$\int_{\R^l} H(x;\lambda, f)\,\nu(dx)\ge\kappa$\, 
for all $f\in\mathbb C_0^2$\,,
then
by \eqref{eq:59}, for arbitrary $\delta>0$\,,
\begin{equation*}
  -\delta\int_{\R^l}\frac{1}{2}\, \text{tr}\,
\bl({\sigma(x)}{\sigma(x)}^T\nabla^2 f(x)\br)\,
\nu(dx)\ge \kappa-\int_{\R^l}\hat H(x;\lambda,-\delta\nabla f(x))
\,\nu(dx)\,.
\end{equation*}
Dividing both sides by $-\delta$ and minimizing the righthand side
over $\delta$
 obtains with the aid of \eqref{eq:65} and the linear growth condition
 \eqref{eq:42} that there exists constant $K_1>0$
such that, for all $f\in\mathbb C_0^2$\,,
\begin{equation*}
    \int_{\R^l} \text{tr}\,
\bl({\sigma(x)}{\sigma(x)}^T\nabla^2 f(x)\br)\,
\,\nu(dx)\le K_1\Bl(\int_{\R^l}\abs{\nabla f(x)}^2\,\nu(dx)\Br)^{1/2}\,.
\end{equation*}
It follows that the lefthand side extends to a linear functional on
$\mathbb L^{1,2}_0(\R^l,\R^l,\nu(dx))$\,, hence, by the Riesz
representation theorem,
there exists $\nabla g\in\mathbb L^{1,2}_0(\R^l,\R^l,\nu(dx))$ such that
\begin{equation}  \label{eq:46}
  \int_{\R^l} \text{tr}\,
\bl({\sigma(x)}{\sigma(x)}^T\nabla^2 f(x)\br)\,
\,\nu(dx)=\int_{\R^l}\nabla g(x)^T\nabla f(x)\,\nu(dx)
\end{equation}
and 
$  \int_{\R^l}\abs{\nabla g(x)}^2\nu(dx)\le K_1\,.
$
Theorem 2.1 in Bogachev, Krylov, and R\"ockner \cite{BogKryRoc01}
 implies that the
measure $\nu(dx)$ has  density $m(x)$ with respect to  the Lebesgue measure
which belongs to  $\mathbb L_{\text{loc}}^\vartheta(\R^l)$ 
for all  $\vartheta<l/(l-1)$\,. 
It follows  that, for 
arbitrary open ball $S$ in $\R^l$\,,
  there exists  $K_2>0$ such that,
for all  
$ \mathbb{C}_0^2$--functions $f$ with support in $S$\,,
\begin{equation*}
  \abs{\int_{S} \text{tr}\,\bl(\sigma(x)\sigma(x)^T\nabla^2f(x)
\br)\,m(x)\,dx}\le K_2\bl(\int_S \abs{\nabla f(x)}^{2\vartheta/(\vartheta-1)}\,dx\br)^{(\vartheta-1)/(2\vartheta)}\,.
\end{equation*}

\noindent 
By Theorem 6.1 in Agmon \cite{Agm59}, 
the  density $m$  belongs to
$\mathbb{W}_{\text{loc}}^{1,\zeta}(\R^l)$ for all $\zeta< 2l/(2l-1)$.
 Furthermore, $\nabla g(x)=-\text{div}(\sigma(x)\sigma(x)^T m(x))/m(x)$ so that
$\sqrt{m}\in \mathbb W^{1,2}(\R^l)$\,.
\qed \end{proof}
\begin{remark}
  Essentially, \eqref{eq:46} signifies that one can integrate by parts on
  the lefthand side, so $m(x)$ has to be weakly differentiable.
\end{remark}

If $m\in\hat{\mathbb P}$\,, then  integration by parts in
\eqref{eq:62}  obtains that, for $f\in\mathbb C^2_0$\,,
\begin{equation}
  \label{eq:12}
 G(\lambda,f,m)=\hat G(\lambda,\nabla f,m)\,,
\end{equation}
where
\begin{align}
  \label{eq:11}
  \hat G(\lambda,\nabla f,m)
=&\int_{\R^l}\Bl(\hat H(x;\lambda,\nabla f(x))
-\frac{1}{2}\, \nabla f(x)^T
\,\frac{\text{div}\,({\sigma(x)}{\sigma(x)}^T\,
 m(x))}{m(x)}\,
\Br)\,m(x)\,dx\,.
\end{align}
(We assume that $0/0=0$\,.)
We will use \eqref{eq:11} in order to define
$\hat G(\lambda,\nabla f,m)$ when $\nabla f\in\mathbb
L^{1,2}_0(\R^l,\R^l,m(x)\,dx)$\,. 
Furthermore, we will use \eqref{eq:12}
 to  extend the definition of $G(\lambda,f,m)$
 to functions   $f\in\mathbb C^1_\ell$\,. It is noteworthy that if $f\in\mathbb C_\ell^1\cap
\mathbb C^2$\,, then
\begin{multline*}
  \lim_{R\to\infty}\int_{x\in\R^l:\,\abs{x}\le R}
\frac{1}{2}\, \text{tr}\,\bl({\sigma(x)}{\sigma(x)}^T\nabla^2 f(x)\br)\,
\,m(x)\,dx\\=
\int_{\R^l}
-\frac{1}{2}\, \nabla f(x)^T
\,\text{div}\,({\sigma(x)}{\sigma(x)}^T\,
 m(x))\,dx\,.
\end{multline*}
  \begin{lemma}
  \label{le:conc}
The function $\hat H(x;\lambda,p)$ is strictly convex in
$(\lambda,p)\in\R_+\times \R^l$\,.
Given $m\in
\hat{\mathbb P}$\,,
the function $\hat G(\lambda,\nabla f,m)$ is
strictly convex  in $(\lambda,\nabla f)\in\R_+\times {\mathbb
  L}^{1,2}_0(\R^l,\R^l,m(x)\,dx)$\,.
The function $G(\lambda,f,m)$ is 
  convex  in $(\lambda,f)\in\R_+\times \mathbb C_b^2$ and
 $\inf_{f\in\mathbb C_0^2}G(\lambda,f,m)$ and
  $F(\lambda)$  tend to $\infty$
superlinearly, as
$\lambda\to\infty$\,.
The function $J^{\text{s}}_q$ is finite and continuous on $\R$\,.
\end{lemma}
\begin{proof} 
The Hessian matrix of $\hat H(x;\lambda,p)$
with respect to $(\lambda,p)$ is given by
\begin{align*}
\hat H_{pp}(x;\lambda,p)&=\frac{1}{1+\lambda}\,
\sigma(x) b(x)^Tc(x)^{-1}b(x)\sigma(x)^T
+\sigma(x)Q_1(x)\sigma(x)^T\,,\\
  \hat H_{\lambda\lambda}(x;\lambda,p)&=\frac{1}{(1+\lambda)^3}\,
\norm{a(x)-r(x)\mathbf1+b(x)\sigma(x)^Tp-b(x)\beta(x)}^2_{c(x)^{-1}}
\\&+\beta(x)^TQ_1(x)\beta(x)\,,\\
\hat H_{\lambda p}(x;\lambda,p)&=-\frac{1}{(1+\lambda)^2}\,
\bl(a(x)-r(x)\mathbf1+b(x)\sigma(x)^Tp-b(x)\beta(x)\br)^T\\&c(x)^{-1}b(x)\sigma(x)^T+\beta(x)^TQ_1(x)\sigma(x)^T\,.
\end{align*}
We show that it is positive definite. More specifically, we prove that 
for all $ z\in\R$ and
$y\in\R^l$ such that $ z^2+\abs{y}^2\not=0$\,,
\begin{equation*}
   z^2\hat H_{\lambda\lambda}(x;\lambda,p)
+y^T\hat H_{pp}(x;\lambda,p) y+
2 z\hat H_{\lambda p}(x;\lambda,p)y>0\,.
\end{equation*}
Since $\hat H_{pp}(x;\lambda,p)$ is a   positive definite matrix by
condition (N),
the latter inequality  holds when $ z=0$\,. Assuming $ z\not=0$\,, we
need to show that
\begin{equation}
  \label{eq:104}
  \hat H_{\lambda\lambda}(x;\lambda,p)
+y^T\hat H_{pp}(x;\lambda,p) y+
2\hat H_{\lambda p}(x;\lambda,p)y>0\,.
\end{equation}
Let, for   $e_1=(v_1(x),w_1(x))$ 
and $e_2=(v_2(x),w_2(x))$\,,
where $v_1(x)\in\R^n\,,w_1(x)\in\R^k\,,v_2(x)\in\R^n\,,
w_2(x)\in\R^k$\,, and $x\in\R^l$\,,  the
inner product be defined by
$e_1\cdot e_2=v_1(x)^Tc(x)^{-1}v_2(x)+
w_1(x)^Tw_2(x)$\,.
By the Cauchy-Schwarz inequality, applied to
$e_1=\bl((1+\lambda)^{-3/2}(a(x)-r(x)\mathbf1+b(x)\sigma(x)^Tp-b(x)\beta(x)),
Q_1(x)\beta(x)\br)$ and
$e_2=((1+\lambda)^{-1/2}b(x)\sigma(x)^Ty,Q_1(x)\sigma(x)^Ty)$\,,
we have that $
  (\hat H_{\lambda p}(x;\lambda,p)y)^2
<y^T\hat H_{pp}(x;\lambda,p) y
\hat H_{\lambda\lambda}(x;\lambda,p)\,,
$ with the inequality being strict because, by condition (N), 
 $Q_1(x)\beta(x)$ is not a scalar multiple of 
$Q_1(x)\sigma(x)^Ty$\,.
Thus, \eqref{eq:104} holds, so 
the function $\hat H(x;\lambda,p)$ is strictly convex
in $(\lambda,p)$ on 
$\R_+\times \R^l$ for all $x\in\R^l$\,.
By  \eqref{eq:11},
 $\hat G(\lambda,\nabla f,m)$ is strictly convex in
$(\lambda,\nabla f)$\,, provided $m\in\hat{\mathbb P}$\,,
and by \eqref{eq:13},
 $G(\lambda,f,m)$ is    convex in
$(\lambda,f)$\,.
Thus, $\inf_{f\in\mathbb C_0^2}G(\lambda,f,m)$ is convex in
$\lambda$\,, so $F(\lambda)$ is convex and, hence, continuous.

By \eqref{eq:65} and \eqref{eq:81}, as $\lambda\to\infty$\,,
\begin{equation*}
  \lim_{\lambda\to\infty}\frac{1}{\lambda^2}\,\inf_{p\in\R^l}
\bl( \hat H(x;\lambda,p)-\frac{1}{2}\,p^T
\,\frac{\text{div}\,(\sigma(x)\sigma(x)^T m(x))}{m(x)}\br)=
\frac{1}{2}\,\norm{\beta(x)}^2_{Q_2(x)}\,.
\end{equation*}
The latter quantity  being positive
 by condition (N)
implies, by \eqref{eq:12} and Fatou's lemma, that 
\begin{equation*}
  \liminf_{\lambda\to\infty}\frac{1}{\lambda^2}\,\inf_{f\in\mathbb C_0^2} 
G(\lambda,f,m)>0\,.
\end{equation*}
Hence, by \eqref{eq:29},
$  \liminf_{\lambda\to\infty}F(\lambda)/\lambda^2>0$\,.
Since $F(\lambda)\to\infty$ superlinearly, as $\lambda\to\infty$\,,
the supremum on the right of \eqref{eq:36} can be taken over the same compact
set of $\lambda$ when the values of $q$ come from a bounded set,
implying
 that $J^{\text{s}}_q$ is finite and continuous. \qed
\end{proof}
\begin{remark}
\label{re:deg}  If $\beta(x)=0$\,, then part 2 of Condition (N) does not
  hold but the proof of Lemma \ref{le:conc} 
still goes through except for the last property
  in that
   $F(\lambda)/\lambda^2$ tends to
  zero as $\lambda\to\infty$\,.
Still,
\begin{equation*}
\liminf_{\lambda\to\infty}\frac{1}{\lambda}\,\inf_{f\in\mathbb C_0^2} 
G(\lambda,f,m)\ge -\int_{\R^l}(r(x)-\alpha(x))m(x)\,dx\,,
\end{equation*}
so that
\begin{equation*}
    \liminf_{\lambda\to\infty}\frac{F(\lambda)}{\lambda}\ge
-\inf_{x\in\R^l}(r(x)-\alpha(x))\,.
\end{equation*}
Consequently, 
if $\inf_{x\in\R^l}(r(x)-\alpha(x))<q$\,, then $-\lambda q-F(\lambda)$ tends to
$-\infty$ as $\lambda\to\infty$, so $\sup_{\lambda\ge0}(-\lambda
q-F(\lambda))$ is attained and $J^{\text{s}}_q$ is finite and continuous. That might not be the case if 
$\inf_{x\in\R^l}(r(x)-\alpha(x))\ge q$\,. For instance, if the functions 
$a(x)$\,, $r(x)$\,, $b(x)$\,, and $\sigma(x)$ are constant,
 $\alpha(x)=0$\,, and $q$ is
small enough, then the derivative of
$-\lambda q-F(\lambda)$  is positive for all $\lambda$ and
$J^{\text{s}}_q=\infty$\,.
As a result, $J^{\text{s}}_q$ might fail to be continuous at 
$\inf_{x\in\R^l}(r(x)-\alpha(x))$\,, although it is rightcontinuous
regardless. \end{remark}
\begin{remark}
The convexity property of $\hat   H(x;\lambda,p)$ could be
expected because,  by \eqref{eq:80},
\begin{equation*}
  \hat   H(x;\lambda,p)= -\sup_{u\in\R^n}\bl(\lambda  M(u,x)
- \frac{1}{2}\,\abs{-\lambda
N(u,x)+\sigma(x)^Tp}^2\br)+
p^T\,\theta(x)\,.
\end{equation*}
\end{remark}
By \eqref{eq:11},
 \eqref{eq:29}, \eqref{eq:62}, and
by the set of the gradients of functions from  $\mathbb C^2_0$
being dense in  ${\mathbb
  L}^{1,2}_0(\R^l,\R^l,m(x)\,dx)$\,, 
\begin{align}
  \label{eq:96}
  F(\lambda)=&  \sup_{m\in \hat{\mathbb P}}  
\inf_{\nabla f\in {\mathbb
               L}^{1,2}_0(\R^l,\R^l,m(x)\,dx)}
\hat G(\lambda,\nabla f,m)\,.
\end{align} 
 Since the matrix $T_\lambda(x)$ is uniformly positive definite,
by \eqref{eq:11}, \eqref{eq:80}, and \eqref{eq:65},
$\hat G(\lambda,\nabla f,m)$
tends
to infinity as the $\mathbb L^2(\R^l,\R^l,m(x)\,dx)$--norm 
of $\nabla f$ tends to infinity. 
Since $\hat G(\lambda,\nabla f,m)$ is strictly convex in $\nabla f$\,,
 the infimum 
in \eqref{eq:96} is attained
at a unique point,
 see, e.g.,
Proposition 1.2 on p.35 in Ekeland and Temam \cite{EkeTem76}.
Furthermore, 
since 
\begin{multline}
  \label{eq:7}
   \sup_{m\in \hat{\mathbb P}}  
\inf_{\nabla f\in {\mathbb
               L}^{1,2}_0(\R^l,\R^l,m(x)\,dx)}
\hat G(\lambda,\nabla f,m)
=\sup_{\nu\in\mathcal{P}}  
\inf_{f\in \mathbb C_0^2}
\int_{\R^l}H(x;\lambda, f)\,\nu(dx)\\=
\sup_{\nu\in\mathcal{P}}  
\inf_{f\in\mathcal{A}_\kappa}
\int_{\R^l}H(x;\lambda, f)\,\nu(dx)
\end{multline}
and, for $\lambda>0$\,, 
 by Lemma \ref{le:sup-comp}, Lemma \ref{le:density_differ} and
 \eqref{eq:12},
 the function
$\inf_{f\in \mathcal{A}_\kappa}
\int_{\R^l}H(x;\lambda, f)\,\nu(dx)
$ is  $\sup$--compact in $\nu$\,, we have that the supremum
in \eqref{eq:96} is attained too, provided $\lambda>0$\,. 

\begin{lemma}
  \label{le:minmax}
Suppose that either $\lambda>0$ or that condition \eqref{eq:45} holds
with $\Phi(x)=0$\,. Then there exists $(f^\lambda,m^\lambda)\in
(\mathbb C^1_\ell\cap \mathbb C^2)\times
\hat{ \mathbb
  P}$ that is  a saddle point of $G(\lambda,f,m)$
as a function of $(f,m)$ so that
\begin{multline}
\label{eq:118}  \inf_{f\in\mathbb C^1_\ell\cap \mathbb C^2}\sup_{m\in\hat{ \mathbb
  P}}G(\lambda,f,m)=\inf_{f\in\mathbb C_b^2}
\sup_{m\in \mathbb
  P}G(\lambda,f,m)=\sup_{m\in \mathbb
  P}\inf_{f\in\mathbb C_0^2}G(\lambda,f,m)
\\=
\sup_{m\in \hat{\mathbb
  P}}\inf_{f\in\mathbb C_\ell^1\cap \mathbb
C^2}G(\lambda,f,m)=F(\lambda)\,,
\end{multline}
with 
the infimum on the leftmost side being attained at $f^\lambda$ and the
supremum on the rightmost side being attained at $m^\lambda$\,.
The function $f^\lambda$ satisfies
the ergodic Bellman equation 
\begin{equation}
  \label{eq:89}
H(x;\lambda,f)=F(\lambda),
\end{equation}
for all $x\in\R^l$\,, 
and $m^\lambda(x)$ is the invariant density of a diffusion:
\begin{multline}
  \label{eq:27}
             \int_{\R^l}
\bl(\nabla  h(x)^T(-\lambda \sigma(x) N( u^\lambda(x),x)+\theta(x)
+\sigma(x)\sigma(x)^T\nabla  f^\lambda(x))
\\+\frac{1}{2}\, \text{tr}\,(\sigma(x)\sigma(x)^T
\,\nabla^2  h(x))
\br)\,  m^\lambda(x)\,dx=0\,,
\end{multline}
for all $h\in\mathbb C_0^2$\,, where
\begin{equation}
  \label{eq:16}
             u^\lambda(x)=\frac{1}{1+\lambda}\,c(x)^{-1}\bl(a(x)-r(x)\mathbf1
+\lambda b(x)\beta(x)+b(x)\sigma(x)^T\nabla
  f^\lambda(x)\br)\,.
\end{equation}
The density $m^\lambda(x)$ may be chosen positive, bounded and  of class $\mathbb
 C^1$\,. The functions $\nabla f^\lambda(x)$ and $m^\lambda(x)$ are specified
uniquely.

In addition, the function 
 $F(\lambda)$ is  strictly convex and
 is continuously differentiable, provided $\lambda>0$\,, and the righthand
 derivative at $\lambda\ge0$ is given by
 \begin{multline}
   \label{eq:58}
     F'_+(\lambda)=\int_{\R^l}\bl(-  M(u^\lambda(x),x)
+\lambda
\abs{N(u^\lambda(x),x)}^2- \nabla f^\lambda(x)^T\sigma(x) 
N(u^\lambda(x),x)\br)
\\ m^\lambda(x)\,dx\,.
\end{multline}
\end{lemma}
\begin{proof}
Since
$\int_{\R^l}
 H(x;\lambda,f)\,\nu(dx)$ is an upper semicontinuous 
and concave function of
 $\nu\in\mathcal{P}$\,,
  for all
$f\in\mathcal{A}_\kappa$\,,
is convex in $f\in\mathcal{A}_\kappa$\,, and
$\int_{\R^l}
 H(x;\lambda,f_\kappa)\,\nu(dx)$ is $\sup$--compact in
 $\nu$ by Lemma \ref{le:sup-comp},
an application of Theorem
7 on p.319 in Aubin and Ekeland \cite{AubEke84} yields 
\begin{multline*}
  \sup_{\nu\in \mathcal{P}}\inf_{f\in\mathbb C_b^2} 
\int_{\R^l} H(x;\lambda,f)\,\nu(dx)=
\sup_{\nu\in \mathcal{P}}\inf_{f\in\mathcal{A}_\kappa} 
\int_{\R^l} H(x;\lambda,f)\,\nu(dx)\\=
\inf_{f\in\mathcal{A}_\kappa}\sup_{\nu\in\mathcal{P}}
\int_{\R^l} H(x;\lambda,f)\,\nu(dx)\ge
\inf_{f\in\mathbb C_b^2}\sup_{\nu\in\mathcal{P}}
\int_{\R^l} H(x;\lambda,f)\,\nu(dx)\,,
\end{multline*} 
 the supremum on the leftmost side being attained at some $\nu^\lambda$\,.
It follows that \begin{equation*}
    \sup_{\nu\in \mathcal{P}}\inf_{f\in\mathbb C_b^2} 
\int_{\R^l} H(x;\lambda,f)\,\nu(dx)=
\inf_{f\in\mathbb C_b^2}\sup_{\nu\in\mathcal{P}}
\int_{\R^l} H(x;\lambda,f)\,\nu(dx)\,.
\end{equation*}

By Lemma \ref{le:density_differ},
\begin{equation*}
  \sup_{\nu\in\mathcal{P}}\inf_{f\in\mathbb C_b^2} 
\int_{\R^l} H(x;\lambda,f)\,\nu(dx)=
\sup_{m\in\mathbb P}\inf_{f\in\mathbb C_0^2} 
\int_{\R^l} H(x;\lambda,f)\,m(x)\,dx
\end{equation*}
and $\nu^\lambda(dx)=m^\lambda(x)\,dx$\,, where
$m^\lambda\in\hat{\mathbb P}$\,,
and, by an approximation argument,
\begin{equation*}
  \sup_{\nu\in\mathcal{P}}
 \int_{\R^l} H(x;\lambda,f)\,\nu(dx)=
\sup_{m\in\mathbb P}
 \int_{\R^l} H(x;\lambda,f)\,m(x)\, dx\,.
\end{equation*}
We obtain that
\begin{equation*}
  \inf_{f\in\mathbb C_b^2}\sup_{m\in \mathbb
  P}G(\lambda,f,m)=\sup_{m\in \mathbb
  P}\inf_{f\in\mathbb C_0^2}G(\lambda,f,m)=\inf_{f\in\mathbb
  C_0^2}G(\lambda,f,m^\lambda)\,. 
\end{equation*}
Therefore, on applying Lemma \ref{le:density_differ},
\begin{multline*}  
  \inf_{f\in\mathbb C^1_\ell\cap \mathbb C^2}\sup_{m\in\hat{ \mathbb
  P}}G(\lambda,f,m)\le  \inf_{f\in\mathbb C_b^2}\sup_{m\in \mathbb
  P}G(\lambda,f,m)=\sup_{m\in \mathbb
  P}\inf_{f\in\mathbb C_0^2}G(\lambda,f,m)\\=
\sup_{m\in \hat{\mathbb
  P}}\inf_{f\in\mathbb C_\ell^1\cap \mathbb C^2}G(\lambda,f,m)\,.
\end{multline*}
The leftmost side not being 
 less than the rightmost side obtains \eqref{eq:118}.

By \eqref{eq:118},
 \begin{equation}
   \label{eq:28}
   F(\lambda)    =\inf_{f\in \mathbb C_b^2}\sup_{x\in\R^l}
 H(x;\lambda, f)\,.
 \end{equation}
(The righthand side is finite: take $f=f_\kappa$\,.)
Applying the reasoning on
 pp.289--294 in Kaise and Sheu \cite{KaiShe06}, one can see that, for
 arbitrary $\epsilon>0$\,,
 there exists $\mathbb C^2$-function $ f^{(\epsilon)}$ 
such that, for all $x\in\R^l$\,,
$    H(x;{\lambda},  f^{(\epsilon)})
=F(\lambda)+\epsilon$\,.
Considering that some details are omitted 
in Kaise and Sheu \cite{KaiShe06}, we give 
an outline of the proof, following the lead of Ichihara \cite{Ich11}.
By the definition of the infimum,
 there exists function $f_1\in \mathbb C_b^2$ such that
$ H(x;\lambda, f_1)\le
F(\lambda)+\epsilon$ 
for all $x$\,. 
Given  open ball $S$\,, centered at the origin,
 by Theorem 6.14 on p.107 in Gilbarg and
Trudinger \cite{GilTru83}, there exists 
$\mathbb C^2$--solution $f_2$ of the linear
elliptic boundary value problem
$ H(x;\lambda, f) -(1/2)\,\nabla f(x)^TT_\lambda(x)\nabla f(x)
=F(\lambda)+\epsilon$ when $x\in S$ and
  $f(x)=f_\kappa(x)$
when $x\in\partial S$\,, with $\partial S$ standing for
the boundary of $S$\,. 
Therefore, 
$ H(x;\lambda, f_2)\ge
F(\lambda)+\epsilon$ in $S$\,.
By Theorem 8.4 on p.302 of Chapter 4 in 
Ladyzhenskaya and Ural'tseva \cite{LadUra68}, for any ball $S'$
contained in $S$ and centered at the origin, there exists 
$\mathbb C^{2}$--solution $f_{S'}$ to
the boundary value problem $ H(x;\lambda, f)=
F(\lambda)+\epsilon$ in $S'$ and $f(x)=f_\kappa(x)$ on $\partial S'$\,.
Since $f_{S'}$ is a solution $f$ of the boundary value problem 
$(1/2)\text{tr}\,(\sigma(x)\sigma(x)^T\nabla^2f(x))=
-\hat H(x;\lambda,\nabla f_{S'}(x))+F(\lambda)+\epsilon$ when $x\in
S'$ and
$f(x)=f_{S'}(x)$ when $x\in \partial S'$\,, we have by Theorem 6.17 on
p.109 of Gilbarg and Trudinger \cite{GilTru83} that 
$f_{S'}(x)$ is thrice continuously differentiable.
Letting the radius of $S'$ (and that of
$S$) go to infinity, we have, by p.294 in Kaise and Sheu
\cite{KaiShe06}, see also Proposition 3.2 in Ichihara \cite{Ich11},
  that the $f_{S'}$ converge locally uniformly and  in 
$\mathbb
W^{1,2}_{\text{loc}}(\R^l)$ to $f^{(\epsilon)}$ which is a weak solution to
$H(x;\lambda,f)=F(\lambda)+\epsilon$\,. 
Furthermore, by Lemma 2.4 in Kaise and Sheu \cite{KaiShe06}, the 
$\mathbb W^{1,\infty}(S'')$--norms of the $f_{S'}$ are uniformly bounded over
balls $S'$ for any fixed ball $S''$ contained in the
$S'$\,. Therefore, $f^{(\epsilon)}$ belongs to $\mathbb
W^{1,\infty}_{\text{loc}}(\R^l)$\,.
By Theorem 6.4 on p.284 in Ladyzhenskaya and Uraltseva
\cite{LadUra68},
$f^{(\epsilon)}$ is thrice continuously differentiable.

As in Theorem 4.2 in Kaise and Sheu
 \cite{KaiShe06}, by using the gradient bound in Lemma 2.4 there, we
 have that the  $f^{(\epsilon)}$ converge along a
 subsequence  uniformly on compact
 sets  as $\epsilon\to0$ to
 a $\mathbb C^2$--solution of $    H(x;{\lambda},  f)
=F(\lambda)$\,.
 That solution, which we denote by $f^\lambda$\,, 
delivers  infimum on the leftmost side of
 \eqref{eq:118}
and satisfies the Bellman equation, with the gradient
 $\nabla f^\lambda(x)$ obeying the linear growth condition,
 see Remark 2.5 in Kaise and Sheu \cite{KaiShe06}.

Since $f^\lambda$ delivers   infimum on the leftmost side of
 \eqref{eq:118} and $m^\lambda$ delivers  supremum on the rightmost side,
 by Proposition
2.156 on p.104 in Bonnans and Shapiro \cite{BonSha00}
or by Proposition 1.2 on p.167 in Ekeland and Temam \cite{EkeTem76},
the pair $(f^\lambda,m^\lambda)$ is  a saddle point of
$G(\lambda,f,m)$ as a function of $(f,m)$\,.
Equation \eqref{eq:27} expresses the requirement of the directional
derivative of $G(\lambda,f,m)$ with respect to $f$ in the direction
$h$ being equal to zero at $(f^\lambda,m^\lambda)$\,, cf. Proposition
1.6 on p.169 in Ekeland and Temam \cite{EkeTem76}.
In some more detail,
 either 
by Theorem 4.13 on p.273 or by Theorem 4.17 on p.276 in Bonnans and Shapiro
\cite{BonSha00}, or by a direct calculation,
the function 
 $\sup_{u\in\R^n}\bl(  M(u,x)
-\lambda
\abs{N(u,x)}^2/2+ p^T\sigma(x) N(u,x)\br)$\,, where $p\in\R^l$\,,
  is differentiable in $p$
with the directional derivative at $p=\nabla f(x)$ in  direction
$\tilde p\in\R^l$ being 
$(\sigma(x) N( \tilde  u(x),x))^T\tilde p$ and with $\tilde u(x)$
representing the point at which 
$\sup_{u\in\R^n}\bl(  M(u,x)
-\lambda
\abs{N(u,x)}^2/2+ \nabla f(x)^T\sigma(x) N(u,x)\br)$
 is attained and being evaluated analogously to \eqref{eq:69}.
It follows from \eqref{eq:62} and dominated convergence that $G(\lambda, f, m)$ 
 has a directional derivative in $f$ in the direction $h$ given by
 \begin{multline}
   \label{eq:98}
   \int_{\R^l}\Bl(\nabla h(x)^T\bl( -\lambda (\sigma(x) N(\tilde u(x),x))+
\theta(x)
+\sigma(x)\sigma(x)\nabla f(x)\br)\\
+\frac{1}{2}\, \text{tr}\,\bl({\sigma(x)}{\sigma(x)}^T\nabla^2 h(x)\br)\,
\Br)\,m(x)\,dx\,.
 \end{multline}
 The function  $\nabla
f^\lambda$ is specified uniquely because
 $ G(\lambda,f,m)$ is a
strictly convex function of $\nabla f$ by \eqref{eq:12},
 cf. Proposition 1.5 on p.169
in Ekeland and Temam \cite{EkeTem76}, and the density
 $m^\lambda$ is specified
uniquely because, by  $\nabla  f^\lambda(x)$
having at most linear  growth in $x$ and by Example 1.7.11
 in Bogachev, Krylov, and R\"ockner \cite{BogKryRoc}, \eqref{eq:27} is
 uniquely solvable for $m^\lambda$\,. 
The function $m^\lambda(x)$ is positive, bounded and is of class $\mathbb
 C^1$ by Corollaries 2.10 and 2.11 in Bogachev, Krylov, and R\"ockner 
\cite{BogKryRoc01} and by
Agmon \cite{Agm59}, see also Theorem 4.1(ii) and p.413 in 
Metafune, Pallardi, and Rhandi \cite{MetPalRha05}.

Since $(f^\lambda,m^\lambda)$ is a saddle point of
$G(\lambda,f,m)$\,, with $m^\lambda$ being specified uniquely,
the suprema in \eqref{eq:96} and on the rightmost side
of \eqref{eq:7} are attained at unique $\nu$ which is $\nu^\lambda$\,.
Both the infimum and supremum in \eqref{eq:96} 
being attained when $\lambda>0$ and  the function $\hat 
G(\lambda,\nabla f,m)$ being
strictly convex in $(\lambda,\nabla f)$\,, the function
$F(\lambda)$ is strictly convex.

We address  the differentiability of $F(\lambda)$\,. By
  Theorem 4.13 on p.273
 in Bonnans and Shapiro
 \cite{BonSha00}   and dominated convergence, we  have on recalling
\eqref{eq:80}
and \eqref{eq:11}, that if $m\in\hat{\mathbb P}$ and $\nabla
f\in\mathbb L^{1,2}_0(\R^l,\R^l,m(x)\,dx)$\,, then
$ \hat G(\lambda,\nabla f,m)$  is differentiable in $\lambda$
 at $\lambda>0$ with the
derivative being equal to
\begin{equation*}
\int_{\R^l}\bl( -  M(\tilde u(x),x)
+\lambda
\abs{N(\tilde u(x),x)}^2- \nabla f(x)^T\sigma(x) N(\tilde u(x),x)\br)
 m(x)\,dx  
\end{equation*}
 and with $\tilde u(x)$ being defined
earlier in this proof.
Furthermore,
$\inf_{\nabla f\in {\mathbb
               L}^{1,2}_0(\R^l,\R^l,m(x)\,dx)}
\hat G(\lambda,\nabla f,m)$
is attained at  $\nabla \tilde f$  such that
the Fr\'echet derivative of $\hat G(\lambda,\nabla f,m)$ with respect
to $\nabla f$\, equals zero, so, by \eqref{eq:80}, \eqref{eq:65}, and
\eqref{eq:98}, 
\begin{multline}
  \label{eq:32}
               \int_{\R^l}
\nabla  h(x)^T\Bl(-\lambda \sigma(x)\Bl(b(x)^T
\frac{1}{1+\lambda}\,c(x)^{-1}\bl(a(x)-r(x)\mathbf1
+\lambda b(x)\beta(x)\br)-\beta(x)\Br) \\+\theta(x)
-\frac{1}{2}\,\frac{ \text{div}\,(\sigma(x)\sigma(x)^T
m(x))}{m(x)}
+T_\lambda(x)
\nabla  \tilde f(x)
\Br)\,  m(x)\,dx=0\,,
\end{multline}
for all $\nabla h\in\mathbb L^{1,2}_0(\R^l,\R^l,m(x)\,dx)$\,.
The mapping that associates with an element $(\lambda,\nabla f)$
of $\R_+\times\mathbb L^{1,2}_0(\R^l,\R^l,m(x)\,dx)$ 
the linear functional  on $\mathbb L^{1,2}_0(\R^l,\R^l,m(x)\,dx)$
that is defined, for  $\nabla h\in\mathbb
L^{1,2}_0(\R^l,\R^l,m(x)\,dx)$\,, by the
lefthand side of \eqref{eq:32}  with $\nabla f$ as $\nabla \tilde
f$\,, is continuously differentiable.
Since the matrix
$T_\lambda(x)$
 is uniformly positive definite, by the Riesz representation theorem,
 the partial derivative with respect to $\nabla f$ 
 is a linear homeomorphism. By the implicit mapping theorem,
$\nabla \tilde f$ is continuously
differentiable in $\lambda$\,, see, e.g., Theorem 2.1 on p.364 in 
Lang \cite{Lan93}. Since
the Fr\'echet derivative of $\hat G(\lambda,\nabla f,m)$ with respect
to $\nabla f$ equals
zero at $\nabla\tilde f$\,, we obtain by
the chain rule that $\hat G(\lambda,\nabla\tilde f,m)$ is continuously
differentiable in $\lambda$ with the full $\lambda$--derivative 
 being  equal to its partial
$\lambda$--derivative
evaluated at $(\lambda,\nabla \tilde f,m)$\,, i.e.,
$\int_{\R^l}\bl( -  M(\tilde u(x),x)
+\lambda
\abs{N(\tilde u(x),x)}^2- \nabla\tilde f(x)^T\sigma(x) N(\tilde u(x),x)\br)
 m(x)\,dx$\,.
 By \eqref{eq:12},
$\inf_{\nabla f\in {\mathbb
               L}^{1,2}_0(\R^l,\R^l,m(x)\,dx)}
\hat G(\lambda,\nabla f,m)=\inf_{f\in\mathbb C^2_0}
\int_{\R^l}H(x;\lambda,f)\,m(x)\,dx$\,, so the latter function is
differentiable in $\lambda$ too, with the same derivative.
This proves \eqref{eq:58} when $\lambda>0$\,.
The case where $\lambda=0$ is obtained by an application of Theorem 24.1
on p.227 in Rockafellar \cite{Rock}.

Given $\overline\lambda>0$\,,
if  $\lambda$ is close enough to
$\overline\lambda$\,, then
\begin{equation}
  \label{eq:10}
  \sup_{\nu\in\mathcal{P}}  
\inf_{f\in \mathbb C_0^2}
\int_{\R^l}H(x;\lambda, f)\,\nu(dx)
=\sup_{\nu\in\mathcal{P}_{\overline\lambda}}  
\inf_{f\in \mathbb C_0^2}
\int_{\R^l}H(x;\lambda, f)\,\nu(dx)\,,
  \end{equation}
where $\mathcal{P}_{\overline\lambda}=
\cup_{\{\breve\lambda:\,\abs{\breve\lambda-\overline\lambda}\le\overline\lambda/2\}}
\{\nu\in \mathcal{P}:\,\inf_{f\in\mathbb C^2_0}\int_{\R^l}H(x;\breve\lambda,
f)\,\nu(dx) \ge F(\overline\lambda)-1\}$\,.
By Lemma \ref{le:density_differ},
the measures from  $\mathcal{P}_{\overline\lambda}$
possess densities which belong to $\hat{\mathbb P}$\,. Hence,  the
function in the supremum on the right of \eqref{eq:10}
is differentiable in $\lambda$ for $\nu\in\mathcal{P}_{\overline\lambda}$\,.
It is also convex in $\lambda$ and upper semicontinuous in $\nu$\,.
By Lemma \ref{le:sup-comp}, the set $\mathcal{P}_{\overline\lambda}$ is
relatively  compact.
In addition,  
$\nu(dx)=m^{\overline\lambda}(x)\,dx$ is the only point at which the supremum on
the lefthand side of \eqref{eq:10} is attained for
$\lambda=\overline\lambda$\,.  
 Theorem 3 on p.201 in Ioffe and
Tihomirov \cite{IofTih79} enables us to conclude that
the righthand side of \eqref{eq:10} is differentiable in $
\lambda$ at $\overline\lambda$\,, 
with the
derivative  being equal to 
\begin{equation*}
  \int_{\R^l}\bl( -  M(u^{\overline\lambda}(x),x)
+\lambda
\abs{N(u^{\overline\lambda}(x),x)}^2
- \nabla f^{\overline\lambda}(x)^T\sigma(x) N(u^{\overline\lambda}(x),x)\br)
 m^{\overline\lambda}(x)\,dx\,.
\end{equation*}
By \eqref{eq:96} and \eqref{eq:7}, this is  true of $F(\lambda)$ too.
\qed
\end{proof}
\begin{remark}
\label{re:diffusion} 
By Theorem 6.4 on p.284 in Ladyzhenskaya and Uraltseva
\cite{LadUra68}, $f^\lambda$ is thrice continuously differentiable.
 Furthermore, by \eqref{eq:28},
$F(\lambda)$
is the smallest $\Lambda$ such that there exists $\mathbb
C^2$--function $f$
 that satisfies
 the equation 
$ H(x;\lambda, f)=\Lambda
$\,, for all $x\in\R^l$\,. 
One can thus infer  the  existence of $m^\lambda(x)$ satisfying
\eqref{eq:27} from the results of
  Kaise and Sheu \cite{KaiShe06} and
Ichihara \cite{Ich11}.
Besides, we have that 
\begin{equation*}
   F(\lambda)    =\inf_{f\in \mathbb C_b^2}\sup_{x\in\R^l}
 H(x;\lambda, f)=
\inf_{f\in \mathbb C_\ell^1}\sup_{x\in\R^l}
 H(x;\lambda, f)=\inf_{f\in \mathbb C^2}\sup_{x\in\R^l}
 H(x;\lambda, f)\,.
\end{equation*}
  \end{remark}

According to Lemma \ref{le:conc} and Lemma \ref{le:minmax},
$\sup_{\lambda\ge0}(-\lambda q-F(\lambda))$
is attained at unique $\lambda$ which is  denoted by $\hat \lambda$\,.
\begin{lemma}
  \label{le:saddle_3}
Suppose that either $\hat\lambda>0$ or condition \eqref{eq:45} holds
with $\Phi(x)=0$\,. Then there exists  $(\hat f,\hat m)\in
 (\mathbb C^1_\ell\cap\mathbb C^2)\times\hat{ \mathbb P}$
such that 
$((\hat\lambda,\hat f),\hat m)$
is a saddle point of 
the function $\lambda q+G(\lambda,f,m)$
in $\R_+\times (\mathbb C^1_\ell\cap\mathbb C^2)\times\hat{ \mathbb
  P}$\,.
Furthermore,  $\nabla\hat f$ and $\hat m$ are specified 
 uniquely. The density $\hat m$ is positive, bounded, and is of class $\mathbb C^1$\,.
Equations \eqref{eq:103'}, \eqref{eq:104'}, and \eqref{eq:69} hold,
 and
\begin{equation}
  \label{eq:2}
    \int_{\R^l}
\bl(  M(\hat u(x),x)
-\hat\lambda\abs{N(\hat u(x),x)}^2+\nabla\hat f(x)^T\sigma(x) N(\hat u(x),x)
\br)\,\hat m(x)\,dx\le q\,,
\end{equation}
with $\hat u(x)$ being defined by \eqref{eq:69}. If $\hat\lambda>0$\,, then
equality prevails in \eqref{eq:2}.

\end{lemma}
\begin{proof}
We let $\hat f=f^{\hat\lambda}$ and $\hat m=m^{\hat\lambda}$\,.
Equations \eqref{eq:103'}, \eqref{eq:104'}, and \eqref{eq:69} hold by Lemma
\ref{le:minmax}.  
Since $G(\lambda,f,m)$ is convex in $(\lambda,f)$ and is concave in $m$\,,
those equations imply that
$((\hat\lambda,\hat f),\hat m)$ is a saddle point of 
$\lambda q+G(\lambda,f,m)$\,, cf. Proposition 1.7 on p.170 in Ekeland
and Temam \cite{EkeTem76}. 
The pair $(\nabla \hat f,\hat m)$ is specified uniquely by  Lemma
\ref{le:minmax}. 
The inequality in  \eqref{eq:2} follows from \eqref{eq:58} and the
fact that $q+F_+'(\hat\lambda)\ge0$\,. If $\hat\lambda>0$\,, then the
latter inequality is equality.

 \qed\end{proof}
\begin{remark}
  \label{re:saddle}
By \eqref{eq:29} and \eqref{eq:36},
under the hypotheses of the theorem, $J^{\text{s}}_q=-\hat\lambda q-
G(\hat\lambda,\hat f,\hat m)$\,.
\end{remark}
 \begin{lemma}
  \label{le:condition}
Suppose that  
\eqref{eq:45} holds with $\Phi(x)=0$\,.
Suppose that either $\hat\lambda=0$ or
there exist $\varrho>0$\,, $C_1>0$ and $C_2>0$ such that
\eqref{eq:31} holds for all $x\in\R^l$\,.
Then \eqref{eq:22} holds.
\end{lemma}
\begin{proof}
By \eqref{eq:18},
\begin{equation}
  \label{eq:20}
\sup_{\nu\in\mathcal{P}}  \inf_{f\in\mathbb C_0^2}
\int_{\R^l}\breve H(x;\hat\lambda,f,u^\tau)\nu(dx)=
\inf_{f\in\mathbb C_b^2}
 \sup_{x\in \R^l}\breve H(x;\hat\lambda,f,u^\tau)\,.
 \end{equation}
For function $f$ and $\tau>0$\,, we denote
$f(x)^\tau=f(x)\chi_{[0,\tau]}(\abs{x})$\,.
By  \eqref{eq:4},  \eqref{eq:8}, \eqref{eq:69}, and \eqref{eq:61}, 
\begin{multline}
  \label{eq:34}
  \breve H(x;\hat\lambda,f,\hat u^\tau)\\=
\frac{\hat\lambda}{2(1+\hat\lambda)}\,\bl(\norm{b(x)\sigma(x)^T\nabla\hat
  f(x)^\tau}^2_{c(x)^{-1}}
-\norm{(a(x)-r(x)\mathbf1)^\tau}^2
_{c(x)^{-1}}\br)
\\-\hat\lambda(r(x)-\alpha(x)
+\frac{1}{2}\,\abs{\beta(x)}^2)
-\frac{\hat\lambda}{2(1+\hat\lambda)}\,
\norm{\hat\lambda b(x)\beta(x)^\tau}^2_{c(x)^{-1}}
\\
-
\frac{\hat\lambda}{1+\hat\lambda}\,\Bl(\bl(\bl(a(x)-r(x)\mathbf1
\br)^\tau\br)^Tc(x)^{-1}b(x)
\hat\lambda\beta(x)
\\+
\bl(\bl(a(x)-r(x)\mathbf1
+\hat\lambda b(x)\beta(x)+b(x)\sigma(x)^T\nabla\hat
  f(x)\br)^\tau\br)^Tc(x)^{-1}b(x){\sigma(x)}^T\nabla f(x)\Br)\\
+\frac{1}{2}\,\abs{\hat\lambda\beta(x)+{\sigma(x)}^T\nabla f(x)}^2
+\nabla f(x)^T\,\theta(x)
+\frac{1}{2}\, \text{tr}\,\bl({\sigma(x)}{\sigma(x)}^T\nabla^2
f(x)\br)\,.
\end{multline}
As in the proof of Lemma \ref{le:sup-comp},
it follows that, under the hypotheses,  there exist
$\breve\kappa>0$\,, $\breve K_1>0$ and $\breve K_2>0$
such that
$ \breve H(x;\hat\lambda,f_{\breve\kappa},\hat u^\tau)\le 
\breve K_1-\breve K_2\abs{x}^2$\,, for all
$x\in\R^l$ and all $\tau>0$\,.
Consequently,  $ \inf_{f\in\mathbb C_0^2}\int_{\R^l}\breve
H(x;\hat\lambda,f,\hat u^\tau)\nu(dx)$ is a $\sup$--compact function
of $\nu\in\mathcal{P}$\,, so,
  the supremum over $\nu$ on the leftthand side of \eqref{eq:20} is attained at
some $\nu_\tau$\,. Moreover, if the $\limsup$ on the lefthand side of
\eqref{eq:22} is greater than $-\infty$\,, then
\begin{equation}
  \label{eq:25}
  \limsup_{\tau\to\infty}\int_{\R^l}\abs{x}^2\nu_\tau(dx)<\infty\,,
\end{equation}
so, the $\nu_\tau$ make up a 
relatively compact subset of $\mathcal{P}$\,.

If either $\hat\lambda=0$ or
\eqref{eq:31} holds, then,
  given $\tilde f\in\mathbb C_0^2$\,, by \eqref{eq:34},
there exist $\tilde C_1$ and $\tilde C_2$\,, such that, for all
$x\in\R^l$ and all $\tau>0$\,,
\begin{equation}
  \label{eq:35}\breve H(x;\hat\lambda,\tilde f,\hat u^\tau)\le
\tilde C_1\abs{x}+\tilde C_2\,.
\end{equation}
Assuming   that $\nu_\tau\to\tilde \nu$\,, we have,
by the convergence  $\breve H(x_\tau;\hat\lambda,\tilde f,\hat
u^\tau)
\to \breve H(\tilde x;\hat\lambda,\tilde f,\hat u)$ when $x_\tau\to\tilde x$\,, by
\eqref{eq:25}, \eqref{eq:35}, the definition of the topology on 
$\mathcal{P}$\,, Fatou's lemma, and the dominated convergence theorem, that
\begin{equation*}
  \limsup_{\tau\to\infty}
\int_{\R^l}\breve H(x;\hat\lambda,\tilde f,\hat u^\tau)\nu_\tau(dx)
\le\int_{\R^l}\breve H(x;\hat\lambda,\tilde f,\hat u)\tilde\nu(dx)\,,
\end{equation*}
so, on recalling \eqref{eq:19},
\begin{equation*}
    \limsup_{\tau\to\infty}
\inf_{f\in\mathbb C_0^2}
\int_{\R^l}\breve H(x;\hat\lambda,f,\hat u^\tau)\nu_\tau(dx)
\le \inf_{f\in\mathbb C_0^2}\int_{\R^l}\breve H(x;\hat\lambda,f,\hat
u)\tilde\nu(dx)\le F(\hat\lambda)\,.
\end{equation*}
\qed
\end{proof}
\section{Proofs of the main results }
\label{sec:proof-bounds}

Proof of Theorem \ref{the:bounds}.
Let $\tilde q<q$\,.
By the continuity of $J^{\text{s}}_q$\,, it suffices to prove that
 if $\tilde q$ is  close enough to $q$\,, then
\begin{equation*}
  \liminf_{t\to\infty}\frac{1}{t}\ln
\mathbf{P}(L^{\pi}_t< q)\ge -J^{\text{s}}_{\tilde q}\,.
\end{equation*}
If $\hat\lambda>0$\,, then $F'(\hat\lambda)= q$\,. Since the function
$F(\lambda)$ is strictly convex, $F'(\lambda)$ is a strictly
increasing function. By it being continuous according to Lemma
\ref{le:minmax} and the intermediate value theorem,
the equation $F'(\lambda)=\tilde q$ has  positive solution
$\tilde\lambda$ provided $\tilde q$ is close enough to $q$\,.
If $\hat\lambda=0$\,, then, by hypotheses, \eqref{eq:45}
  holds with $\Phi(x)=0$ and we let $\tilde\lambda=0$\,.
In either case, Lemma
\ref{le:saddle_3} yields the existence
of saddle point $(f^{\tilde \lambda}, m^{\tilde\lambda})$
of the function $G(\tilde\lambda,f,m)$ 
 such that
$(f^{\tilde \lambda}, m^{\tilde\lambda})\in
 (\mathbb C^1_\ell\cap\mathbb C^2)\times\hat{ \mathbb P}$\,. In
 addition, the density $ m^{\tilde\lambda}$ is continuously differentiable, 
positive and bounded.
By Remark \ref{re:saddle},
$  J^{\text{s}}_{\tilde q}= -\tilde\lambda\tilde q-
G(\tilde\lambda,f^{\tilde \lambda}, m^{\tilde\lambda})\,.$
 Therefore, one needs 
 to prove that
\begin{equation}
  \label{eq:48}
            \liminf_{t\to\infty}\frac{1}{t}\ln
\mathbf{P}(L^{\pi}_t< q)\ge 
\tilde\lambda \tilde q+
G(\tilde\lambda,f^{\tilde \lambda}, m^{\tilde\lambda})\,.
\end{equation}
Let $ u^{\tilde\lambda}(x)$ represent the point at which the supremum is attained
in \eqref{eq:62} for $\lambda=\tilde \lambda$\,,
$m= m^{\tilde\lambda}$\,, $f=f^{\tilde \lambda}$ so that, as in \eqref{eq:69} and \eqref{eq:16},
\begin{equation}
  \label{eq:33}
           u^{\tilde\lambda}(x)=\frac{1}{1+\tilde\lambda}\,c(x)^{-1}\bl(a(x)-r(x)\mathbf1
+\tilde\lambda b(x)\beta(x)+b(x)\sigma(x)^T\nabla f^{\tilde\lambda}(x)\br)\,.
\end{equation}
We prove \eqref{eq:48} by showing that
\begin{equation}
  \label{eq:57}
-\tilde \lambda \tilde q- 
 G(\tilde\lambda,f^{\tilde \lambda}, m^{\tilde\lambda})
=\frac{1}{2}\,\int_{\R^l}\abs{-\tilde\lambda
N( u^{\tilde\lambda}(x),x)+
\sigma(x)^T\nabla  f^{\tilde \lambda}(x)
}^2 m^{\tilde\lambda}(x)\,dx
\end{equation}
and that, for $\delta>0$\,,
  \begin{multline}
    \label{eq:90}
       \liminf_{t\to\infty}
\frac{1}{t}\,\ln     \mathbf{P}\bl(L^{\pi}_t< 
   \tilde q+3\delta\br)\\\ge
-\frac{1}{2}\,\int_{\R^l}\abs{-\tilde\lambda N( u^{\tilde\lambda}(x),x)+
\sigma(x)^T\nabla  f^{\tilde \lambda}(x)
}^2 m^{\tilde\lambda}(x)\,dx-2\delta\,.
  \end{multline}
\begin{itemize}
\item{\bf  Proof of  \eqref{eq:57}.}
  By Lemma \ref{le:saddle_3},
\begin{multline}
  \label{eq:21}
    \int_{\R^l}
\bl(  M( u^{\tilde\lambda}(x),x)
-\tilde\lambda\abs{N( u^{\tilde\lambda}(x),x)}^2+
\nabla f^{\tilde \lambda}(x)^T\sigma(x) N( u^{\tilde\lambda}(x),x)\br)
 m^{\tilde\lambda}(x)\,dx\\\le \tilde q\,,
\end{multline} 
with  equality prevailing,  provided $\tilde\lambda>0$\,.
Therefore, no matter $\tilde\lambda$\,, by \eqref{eq:11} and \eqref{eq:80},
\begin{multline}
  \label{eq:30}
    -\tilde \lambda \tilde q-G(\tilde\lambda,f^{\tilde \lambda},m)=
-\tilde \lambda \tilde q-\hat G(\tilde\lambda,\nabla f^{\tilde\lambda},m)\\=
-\tilde\lambda\int_{\R^l}
\bl(  M( u^{\tilde\lambda}(x),x)
-\tilde\lambda\abs{N( u^{\tilde\lambda}(x),x)}^2\\+
\nabla  f^{\tilde\lambda}(x)^T\sigma(x) N( u^{\tilde\lambda}(x),x)
\br)
 m^{\tilde\lambda}(x)\,dx\\
-\int_{\R^l}
\bl( -\tilde\lambda M( u^{\tilde\lambda}(x),x)
+\frac{1}{2}\,\tilde\lambda^2\abs{N( u^{\tilde\lambda}(x),x)}^2-\tilde\lambda\,
\nabla  f^{\tilde\lambda}(x)^T\sigma(x) N( u^{\tilde\lambda}(x),x)
\\+\nabla  f^{\tilde\lambda}(x)^T\theta(x)
+\frac{1}{2}\,\abs{\sigma(x)^T\nabla  f^{\tilde\lambda}(x)}^2
\\-\nabla  f^{\tilde\lambda}(x)^T\,\frac{\text{div}\bl(
\nabla  f^{\tilde\lambda}(x)^T{\sigma(x)}{\sigma(x)}^T m^{\tilde\lambda}(x)\br)}{2 m^{\tilde\lambda}(x)}\,
\br)
 m^{\tilde\lambda}(x)\,dx\\
=\int_{\R^l}
\frac{1}{2}\,\tilde\lambda^2\abs{N( u^{\tilde\lambda}(x),x)}^2
 m^{\tilde\lambda}(x)\,dx
\\-\int_{\R^l}\bl(\nabla f^{\tilde\lambda}(x)^T\theta(x)+\frac{1}{2}\,\abs{\sigma(x)^T\nabla  f^{\tilde\lambda}(x)}^2
\\-\nabla  f^{\tilde\lambda}(x)^T\,\frac{\text{div}\bl(
{\sigma(x)}{\sigma(x)}^T m^{\tilde\lambda}(x)\br)}{2 m^{\tilde\lambda}(x)}\,
\br)
 m^{\tilde\lambda}(x)\,dx\,.
\end{multline}
By \eqref{eq:27} in
 Lemma \ref{le:saddle_3}, the inclusion $ m^{\tilde\lambda}\in\hat{\mathbb
   P}$\,,
 and integration by parts, for $h\in\mathbb C_0^2$\,,
\begin{multline}
  \label{eq:71}
                     \int_{\R^l}
\nabla  h(x)^T\bl(-\tilde\lambda \sigma(x) N( u^{\tilde\lambda}(x),x)+\theta(x)
+\sigma(x)\sigma(x)^T\nabla f^{\tilde\lambda}(x)
\\-\frac{\text{div}\bl(
\sigma(x)\sigma(x)^T m^{\tilde\lambda}(x)\br)}{2 m^{\tilde\lambda}(x)}\,
\br)\,   m^{\tilde\lambda}(x)\,dx=0\,.
\end{multline}
The facts that $\abs{\nabla  f^{\tilde\lambda}(x)}$ grows at most linearly
 with $\abs{x}$\,, that $ u^{\tilde\lambda}(x)$ is a linear function of $\nabla
 f^{\tilde\lambda}(x)$ by \eqref{eq:33}, 
that $\int_{\R^l}\abs{x}^2\, m^{\tilde\lambda}(x)\,dx<\infty$\,, and that
$\int_{\R^l}\abs{\nabla  m^{\tilde\lambda}(x)}^2/ m^{\tilde\lambda}(x)\,dx<\infty$\,,
imply that the expression in parentheses under the integral in
\eqref{eq:71} represents a function which is an element of 
$\mathbb L^2(\R^l,\R^l, m^{\tilde\lambda}(x)\,dx)$\,. 
Therefore, the integral on the lefthand side extends to a linear
functional on $\mathbb L^{1,2}_0(\R^l,\R^l, m^{\tilde\lambda}(x)\,dx)$\,. 
Consequently, one can substitute $\nabla  f^{\tilde\lambda}(x)$ for $\nabla h(x)$ to obtain that
\begin{multline}
  \label{eq:47a}
                       \int_{\R^l}
\nabla f^{\tilde\lambda}(x)^T\bl(-\tilde\lambda \sigma(x) N( u^{\tilde\lambda}(x),x)+\theta(x)
+\sigma(x)\sigma(x)^T\nabla f^{\tilde\lambda}(x)
\\-\frac{\text{div}\bl(
\sigma(x)\sigma(x)^T m^{\tilde\lambda}(x)\br)}{2 m^{\tilde\lambda}(x)}\,
\br)\,   m^{\tilde\lambda}(x)\,dx=0\,.
\end{multline}
Substitution on the rightmost side of \eqref{eq:30} yields \eqref{eq:57}.
\item  {\bf Proof of \eqref{eq:90}.}
We apply a Girsanov change of a probability measure.  Let
$\tilde{ W}^t_s$ for $s\in[0,1]$  and measure $\tilde{\mathbf{P}}^{t}$ 
be defined by the respective equations
\begin{equation}
  \label{eq:34'}
\tilde{ W}^t_s=  W^t_s-\sqrt{t}\int_0^s(
-\tilde\lambda N( u^{\tilde\lambda}(X^t_{\tilde
s}),X^t_{\tilde s})
+\sigma (X^t_{\tilde s})^T\nabla   f^{\tilde\lambda}(X^t_{\tilde s}) )\,d\tilde s
\end{equation}
and 
\begin{multline}
  \label{eq:35'}
  \frac{d\tilde{\mathbf{P}}^{t}}{d\mathbf{P}}=
\exp\bl(\sqrt{t}\,\int_0^1
(-\tilde\lambda N( u^{\tilde\lambda}(X^t_s),X^t_s)+\sigma(X^t_s)^T\nabla
 f^{\tilde\lambda}(X^t_s)
)^T\, 
d W^t_s\\-
\frac{t}{2}\,\int_0^1\abs{
-\tilde\lambda N
( u^{\tilde\lambda}(X^t_s),X^t_s)+{\sigma(X^t_s)}^T\nabla  f^{\tilde\lambda}(X^t_s)}^2\,ds\br)\,.
\end{multline}
A multidimensional extension of
  Theorem 4.7 on p.137 in
 Liptser and Shiryayev \cite{LipShi77}, which is proved similarly, obtains  that, given $t>0$\,,
 there exists $\gamma'>0$ such that
$\sup_{s\le t}
\mathbf Ee^{\gamma'\abs{X_s}^2}<\infty$\,.  By Example 3 on pp.220,221 in
 Liptser and Shiryayev \cite{LipShi77}
 and  $\nabla  f^{\tilde\lambda}(x)$ obeying the linear growth 
condition, the expectation of the righthand side of
\eqref{eq:35'} with respect to $\mathbf P$ equals unity.
 Therefore, $\tilde{\mathbf{P}}^{t}$
 is a valid 
probability measure and 
the process $(\tilde{ W}^t_s,\,s\in[0,1])$ is a standard Wiener process
under $\tilde{\mathbf{P}}^{t}$\,, see Lemma 6.4 on p.216 in
Liptser and Shiryayev \cite{LipShi77} and Theorem 5.1 on p.191 in
Karatzas and Shreve \cite{KarShr88}.

As a stepping-stone to the proof of  \eqref{eq:90}, 
we establish the ergodic property that
if $\abs{g(x)}\le K(1+\abs{x}^2)$\,, for some $K>0$\,, then, for
arbitrary $\epsilon>0$\,,
\begin{equation}
  \label{eq:48'}
  \lim_{t\to\infty}
\tilde{\mathbf P}^{t}\Bl(\abs{\int_{\R^l}g(x)
\nu_t(dx)-\int_{\R^l}g(x) m^{\tilde\lambda}(x)\,dx}>\epsilon\Br)=0\,,
\end{equation}
where we let
\begin{equation}
  \label{eq:24}
  \nu_t(dx)=\mu_t([0,1],dx)\,.
\end{equation}
 By \eqref{eq:14} and \eqref{eq:34'},
\begin{multline*}
        dX^t_s=t\,\theta(X^t_s)\,ds+t\,
\sigma(X^t_s)\,\bl( -\tilde\lambda N( u^{\tilde\lambda}(X^t_s),X^t_s)+
\sigma(X^t_s)^T\nabla   f^{\tilde\lambda}(X^t_{ s})
 \br)
\,ds\\+\sqrt{t}\sigma(X^t_s) d\tilde{ W}^t_s\,.
\end{multline*}
Hence, the process $X=(X_s\,,s\ge0)=(X^t_{s/t}\,,s\ge0)$ satisfies the
equation
\begin{multline*}
        d X_s=
\theta(X_s)\,ds+
\sigma( X_s)\,\bl( -\tilde\lambda N( u^{\tilde\lambda}( X_s), X_s)+
\sigma( X_s)^T\nabla   f^{\tilde\lambda}( X_{ s})
 \br)
\,ds\\
+\sigma(X_s) d\breve W^t_s\,,
\end{multline*}
$(\breve W_s^t)$ being a standard Wiener process under $\tilde{\mathbf P}^{t}$\,.
By Theorem 10.1.3 on p.251 in Stroock and Varadhan
\cite{StrVar79} the distribution of $X$ under $\tilde{\mathbf P}^{t}$ is
specified uniquely. 
In addition, 
by Theorem 9.1.9 on p.220 and Lemma  9.2.2 on p.234 in Stroock and Varadhan
\cite{StrVar79}, $X$ is a regular Feller process. (See p.399 in
Kallenberg \cite{Kal02} for the definition.)
By \eqref{eq:47a}, $X$ has 
the invariant distribution
  $ m^{\tilde\lambda}(x)\,dx$\,, see, e.g., Theorem 1.5.13 in
Bogachev, Krylov, and R\"ockner \cite{BogKryRoc}.
The process $X$ is therefore positive Harris recurrent, see Theorem
20.17 on p.405 and Theorem 20.20 on p.408 
 in Kallenberg \cite{Kal02}. Since
$ m^{\tilde\lambda}\in\mathbb{\hat P}$\,, we have that
$\int_{\R^l}\abs{x}^2 m^{\tilde\lambda}(x)\,dx<\infty$\,, so
$\int_0^t g(X_s)\,ds$ is an integrable
 additive functional.
The limit in \eqref{eq:48'} now follows by
 Theorem 3.12 on p.397 and the discussion on p.398 in Revuz and Yor
 \cite{RevYor91}.

With the proof of \eqref{eq:48'} being out of the way,
we mount a final assault on \eqref{eq:90}.
By \eqref{eq:34'},
\begin{multline}
  \label{eq:6}
    \int_0^1M(\pi^t_s,X^t_s)\,ds+\frac{1}{\sqrt{t}}\,\int_0^1 
N(\pi^t_s,X^t_s)^T\,d W^t_s=
\int_0^1M(\pi^t_s,X^t_s)\,ds\\+
\int_0^1N(\pi^t_s,X^t_s)^T (
-\tilde\lambda N( u^{\tilde\lambda}(X^t_s
),X^t_s)
+\sigma (X^t_s)^T\nabla   f^{\tilde\lambda}(X^t_s) )\,ds\\+
\frac{1}{\sqrt{t}}\,\int_0^1 
N(\pi^t_s,X^t_s)^T\,d \tilde{ W}^t_s=\frac{1}{t}\,\ln\mathcal{E}_1^t
+
\int_0^1M( u^{\tilde\lambda}(X^t_s),X^t_s)\,ds\\+
\int_0^1N( u^{\tilde\lambda}(X^t_s),X^t_s)^T (
-\tilde\lambda N( u^{\tilde\lambda}(X^t_s
),X^t_s)
+\sigma (X^t_s)^T\nabla   f^{\tilde\lambda}(X^t_s) )\,ds\\+
\frac{1}{\sqrt{t}}\,\int_0^1 
N( u^{\tilde\lambda}(X^t_s),X^t_s)^T\,d \tilde{ W}^t_s\,,
\end{multline}
where
 $\mathcal{E}_s^t$ represents the stochastic exponential defined by
\begin{equation*}
  \mathcal{E}_s^t=\exp\bl(
\sqrt{t}\,\int_0^s(\pi^t_{\tilde s}- u^{\tilde\lambda}(X^t_{\tilde
  s}))^Tb(X^t_{\tilde s})
\,d\tilde{ W}^t_{\tilde s}\\-
\frac{t}{2}\,\int_0^s
\norm{\pi^t_{\tilde s}- u^{\tilde\lambda}(X^t_{\tilde s})}_{c(X^t_{\tilde s})}^2 d\tilde s\br)\,.
\end{equation*}
Since $
\tilde{\mathbf E}^{t}\mathcal{E}_1^t\le 1$\,, Markov's inequality yields
the convergence
\begin{equation}
  \label{eq:38}
  \lim_{t\to\infty}
\tilde{\mathbf P}^{t}\bl(  \frac{1}{t}\ln\mathcal{E}^t_1<\delta\br)=1\,.
\end{equation}

By  \eqref{eq:35'} and \eqref{eq:6}, 
\begin{multline}
  \label{eq:17}
\mathbf{P}\bl(L^{\pi}_t< 
\tilde    q+3\delta\br)  =
\tilde{\mathbf{E}}^{t}\chi_{
\displaystyle\{
\int_0^1
M(\pi_{s}^t,X^t_s)\,ds
+\frac{1}{\sqrt{t}}\,\int_0^1 N(\pi_{s}^t,X_s^t)^T\,dW_s^t}\\{< 
\tilde q+3\delta
\}}\\
\exp\bl(-\sqrt{t}\int_0^1(-\tilde\lambda
 N( u^{\tilde\lambda}(X^t_s),X^t_s)+\sigma( X^t_s)^T\nabla   f^{\tilde\lambda}(X^t_s)
)^T 
\,d\tilde{ W}^t_s
\\-\frac{t}{2}\,
\int_0^1\abs{-\tilde\lambda
 N( u^{\tilde\lambda}(X^t_s),X^t_s)+\sigma(X^t_s)^T\nabla   f^{\tilde\lambda}(X^t_s)}^2\,ds\br)\\\ge
\tilde{\mathbf{E}}^{t}\chi_{\Big\{\displaystyle
\frac{1}{t}\ln\mathcal{E}^t_1<\delta\Big\}}\,
\chi_{\Big\{\displaystyle
\frac{1}{\sqrt{t}}\,\abs{\int_0^1 
N( u^{\tilde\lambda}(X^t_s),X^t_s)^T\,d \tilde{ W}^t_s}<\delta\Big\}}\\
\chi_{\Big\{\displaystyle
\int_0^1\Bl(M( u^{\tilde\lambda}(X^t_s),X^t_s)+N( u^{\tilde\lambda}(X^t_s),X^t_s)^T (
-\tilde\lambda N( u^{\tilde\lambda}(X^t_s
),X^t_s)
}\\{+\sigma (X^t_s)^T\nabla   f^{\tilde\lambda}(X^t_s) )\Br)\,ds<\tilde q 
+\delta\Big\}}
\\
\chi_{\Big\{\displaystyle
\frac{1}{\sqrt t}\,\abs{\int_0^1(
-\tilde\lambda N( u^{\tilde\lambda}(X^t_s),X^t_s)+\sigma(X^t_s)^T\nabla   f^{\tilde\lambda}(X^t_s)
 )^T\,d\tilde{ W}^t_s}
< \delta\Big\}}\\
\chi_{\Big\{\displaystyle
\int_0^1\abs{-\tilde\lambda N( u^{\tilde\lambda}(X^t_s),X^t_s)
+\sigma(X^t_s)^T\nabla 
   f^{\tilde\lambda}(X^t_s)}^2\,ds}\\{-
\int_{\R^l}\abs{-\tilde\lambda N( u^{\tilde\lambda}(x),x)
+\sigma(x)^T\nabla   f^{\tilde\lambda}(x)}^2
 m^{\tilde\lambda}(x)\,dx<2\delta\Big\}}
\\\exp\bl(-2\delta t
-\frac{t}{2}\,\int_{\R^l}\abs{-\tilde\lambda N( u^{\tilde\lambda}(x),x)+
\sigma(x)^T\nabla   f^{\tilde\lambda}(x)
}^2 m^{\tilde\lambda}(x)\,dx\br)\,.
\end{multline}
We work with the terms on the righthand side of \eqref{eq:17} in order.
By \eqref{eq:42}, \eqref{eq:4}, \eqref{eq:8}, \eqref{eq:33},
 \eqref{eq:48'}, and by  $\nabla  f^{\tilde\lambda}(x)$ satisfying the linear growth
condition, 
\begin{multline}
  \label{eq:67}
    \lim_{t\to\infty}\tilde{\mathbf P}^{t}
\bl(\big|\int_0^1\abs{-\tilde\lambda N( u^{\tilde\lambda}(X^t_s),X^t_s)+
\sigma(X^t_s)^T\nabla 
    f^{\tilde\lambda}(X^t_s)}^2\,ds\\-
\int_{\R^l}\abs{-\tilde\lambda N( u^{\tilde\lambda}(x),x)+\sigma(x)^T\nabla   f^{\tilde\lambda}(x)}^2
 m^{\tilde\lambda}(x)\,dx\big|<2\delta\br)=1\,.
\end{multline}
Similarly, by \eqref{eq:48'} and \eqref{eq:21},
\begin{multline*}
  \lim_{t\to\infty}\tilde{\mathbf P}^{t}\bl(
\int_0^1\bl(M( u^{\tilde\lambda}(X^t_s),X^t_s)+N( u^{\tilde\lambda}(X^t_s),X^t_s)^T (
-\tilde\lambda N( u^{\tilde\lambda}(X^t_s
),X^t_s)
\\+\sigma (X^t_s)^T\nabla   f^{\tilde\lambda}(X^t_s) )\br)\,ds<
\tilde q+\delta\br)=1\,.
\end{multline*}

Since,  for $\epsilon>0$\,, by the L\'englart--Rebolledo inequality, see
Theorem 3 on p.66 in
Liptser and Shiryayev \cite{lipshir}, 
\begin{multline*}
 \tilde{\mathbf{P}}^{t}\bl(\abs{\frac{1}{\sqrt{t}}\,
\int_0^1(-\tilde\lambda N( u^{\tilde\lambda}(X^t_s),X^t_s)+
\sigma(x)^T\nabla   f^{\tilde\lambda}(X^t_s) )\,d\tilde{ W}^t_s}\ge
\delta\br)
\\\le 
\frac{\epsilon
}{\delta^2 }+\tilde{\mathbf{P}}^{t}\bl(
\int_0^1\abs{
-\tilde\lambda N( u^{\tilde\lambda}(X^t_s),X^t_s)+
\sigma(X^t_s)^T\nabla   f^{\tilde\lambda}(X^t_s)
}^2\,ds\ge\epsilon t\br)\,,
\end{multline*}
\eqref{eq:67} implies that
\begin{equation*}
\lim_{t\to\infty}
\tilde{\mathbf{P}}^{t}\bl(\frac{1}{\sqrt{t}}\,
\abs{\int_0^1(
-\tilde\lambda N( u^{\tilde\lambda}(X^t_s),X^t_s)+
\sigma(X^t_s)^T\nabla   f^{\tilde\lambda}(X^t_s)
 )\,d\tilde{ W}^t_s}< \delta\br)
=1\,. 
\end{equation*}
Similarly,
\begin{equation*}
\lim_{t\to\infty}
\tilde{\mathbf{P}}^{t}\bl(
\frac{1}{\sqrt{t}}\,\abs{\int_0^1 
N( u^{\tilde\lambda}(X^t_s),X^t_s)^T\,d \tilde{ W}^t_s}
<\delta\br)=1\,.
\end{equation*}
Letting $t\to\infty$ in \eqref{eq:17} and recalling \eqref{eq:38}
obtains \eqref{eq:90}.
\qed
\end{itemize}
\begin{remark}
  The change of measure in \eqref{eq:35'} is implicit in
 Puhalskii \cite{Puh16}. The idea of using a stochastic exponential
in order to
  ''absorb'' control, as in
  \eqref{eq:6}, is borrowed from Hata, Nagai, and Sheu \cite{Hat10}.
\end{remark}

Proof of Theorem \ref{the:risk-sens}.
We start with proving part 1. 
By Theorem \ref{the:bounds}, it suffices to prove that
\begin{equation}
  \label{eq:41}
  \limsup_{\tau\to\infty}\limsup_{t\to\infty}
\frac{1}{t}\,\ln\mathbf P(L^{\hat\pi^\tau}_t\le q)\le -J^{\text{s}}_q\,.
\end{equation}
Let $f\in\mathcal{A}_\kappa$\,.
In analogy with \eqref{eq:1a},
\begin{equation*}
  \mathbf{E}\exp\bl(-t\hat\lambda L^{\hat\pi,\tau}_t+f(X_t)-f(X_0)
-t\int_{\R^l}\breve H(x;\hat\lambda,f,u^\tau)\nu(dx)\br)\le1.
\end{equation*}
Thanks to Jensen's inequality,
\begin{equation*}
  \mathbf{E}
\chi_{\{L^{\hat\pi,\tau}_t\le q\}}\exp(f(X_t)-f(X_0))\le e^{t\hat\lambda q}
\exp\bl(t\sup_{x\in \R^l}\breve H(x;\hat\lambda,f,u^\tau)\br)\,.
\end{equation*}
By reverse H\"older's inequality,
for $\epsilon\in(0,1)$\,,
\begin{equation*}
  \mathbf{E}\chi_{\{L^{\hat\pi,\tau}_t\le q\}}\exp(f(X_t)-f(X_0))
\ge \mathbf P(L^{\hat\pi,\tau}_t\le q)^{1+\epsilon}\bl(\mathbf E
\exp(-\frac{1}{\epsilon}\,(f(X_t)-f(X_0)))\br)^{-\epsilon}\,.
\end{equation*}
Since $-k\le f(x)\le \kappa\abs{x}^2+k$\,, for some $k>0$\,,
if $\kappa<\epsilon\gamma$\,, then by \eqref{eq:37}
$  \lim_{t\to\infty}\bl(\mathbf E
\exp(-(f(X_t)-f(X_0))/\epsilon)\br)^{1/t}\le1\,,
$  which implies that
\begin{equation*}
    \limsup_{t\to\infty}
\,\frac{1+\epsilon}{t}\,\ln\mathbf{P}(L^{\hat\pi,\tau}_t\le q)\le
\hat\lambda q+
\inf_{f\in\mathcal{A}_\kappa}
\sup_{x\in \R^l}\breve H(x;\hat\lambda,f,u^\tau)\,.
\end{equation*}
In analogy with the proof of Lemma \ref{le:sup-comp},
\begin{multline}
  \label{eq:18}
 \inf_{f\in\mathcal{A}_\kappa}
\sup_{x\in \R^l}\breve H(x;\hat\lambda,f,u^\tau)=
  \inf_{f\in\mathcal{A}_\kappa}
\sup_{\nu\in\mathcal{P}}\int_{\R^l}\breve H(x;\hat\lambda,f,u^\tau)\nu(dx)\\=
\sup_{\nu\in\mathcal{P}}
  \inf_{f\in\mathcal{A}_\kappa}
\int_{\R^l}\breve H(x;\hat\lambda,f,u^\tau)\nu(dx)=
\sup_{\nu\in\mathcal{P}}  \inf_{f\in\mathbb C^2_b}
\int_{\R^l}\breve H(x;\hat\lambda,f,u^\tau)\nu(dx)\\\le
\inf_{f\in\mathbb C^2_b}\sup_{\nu\in\mathcal{P}  }
\int_{\R^l}\breve H(x;\hat\lambda,f,u^\tau)\nu(dx)
=\inf_{f\in\mathbb C^2_b}\sup_{x\in\R^l  }
\breve H(x;\hat\lambda,f,u^\tau)\,.\end{multline}
Hence,
\begin{equation*}
  \inf_{f\in\mathcal{A}_\kappa}
\sup_{x\in \R^l}\breve H(x;\hat\lambda,f,u^\tau)
=\inf_{f\in\mathbb C^2_b}\sup_{x\in\R^l}\breve H(x;\hat\lambda,f,u^\tau)\,,
\end{equation*}
so, by $\epsilon$ being arbitrarily small,
\begin{equation*}
    \limsup_{t\to\infty}
\,\frac{1}{t}\,\ln\mathbf{P}(L^{\hat\pi,\tau}_t\le q)\le
 \hat\lambda q+
\inf_{f\in\mathbb C^2_b}\sup_{x\in\R^l}\breve H(x;\hat\lambda,f,u^\tau)
\end{equation*}
and the required property follows by \eqref{eq:22}. 

We prove now part 2.
By Theorem \ref{the:bounds}, it suffices to prove that
\begin{equation}
  \label{eq:41a}
  \limsup_{t\to\infty}
\frac{1}{t}\,\ln\mathbf P(L^{\hat\pi}_t\le q)\le -J^{\text{s}}_q\,.
\end{equation}
We borrow from Koncz \cite{Kon87} and
 Nagai \cite{Nag03}.
Similarly to the proof of Theorem \ref{the:bounds}, we introduce the
change of measure 
\begin{multline*}
  \frac{d\hat{\mathbf{P}}}{d\mathbf{P}}\Big|_{\mathcal{F}_t}=
\exp\bl(\int_0^t
(-\hat\lambda N( \hat u(X_s),X_s)+\sigma(X_s)^T\nabla
 \hat f(X_s))^T\, 
d W_s\\-
\frac{1}{2}\,\int_0^t\abs{
-\hat\lambda N
( \hat u(X_s),X_s)+{\sigma(X_s)}^T\nabla\hat  f(X_s)}^2\,ds\br)\,.
\end{multline*}
Then 
$(\hat{ W}_t\,,t\ge0)$ is a standard Wiener process with respect to   
$\hat{\mathbf{P}}$\,, where
\begin{equation*}
\hat{ W}_t=  W_t-\int_0^t\bl(-\hat\lambda N(\hat u(X_s),X_s)+
{\sigma(X_s)}^T\nabla  \hat f(X_s)\br)\,ds\,.
\end{equation*}
By \eqref{eq:14} and It\^o's lemma, 
\begin{equation*}
  dX_t=\bl(\theta (X_t)+\sigma(X_t)(-\hat\lambda N(\hat u(X_s),X_s)+
{\sigma(X_s)}^T\nabla  \hat f(X_s))\br)\,dt+\sigma(X_t)\,d\hat W_t
\end{equation*}
and
\begin{multline*}
  d\hat f(X_t)=\bl(
\nabla \hat f(X_t)^T (\theta (X_t)+\sigma(X_t)(-\hat\lambda N(\hat u(X_t),X_t)+
{\sigma(X_t)}^T\nabla  \hat f(X_t))\\+\frac{1}{2}\,
\text{tr}(\sigma(X_t)\sigma(X_t)^T\nabla^2\hat f(X_t))\br)\,dt
+\nabla \hat f(X_t)^T\sigma(X_t)\,d\hat W_t\,.
\end{multline*}
By \eqref{eq:1},
\eqref{eq:4}, \eqref{eq:8}, and \eqref{eq:103'},
\begin{equation}
  \label{eq:40}
  \mathbf E e^{-t\hat\lambda L^{\hat \pi}_t}=
e^{tF(\hat\lambda)}\mathbf{ \hat E}e^{\hat f(X_0)-\hat f(X_t)}\,.
\end{equation}
By  It\^o's lemma,  \eqref{eq:103'}, and \eqref{eq:61},
\begin{multline*}
  e^{\hat f(X_0)-\hat f(X_t)}=1+
\int_0^t e^{\hat f(X_0)-\hat f(X_s)}\bl(-\nabla \hat f(X_s)^T\theta (X_s)\\
+\hat\lambda \nabla\hat  f(X_s)^T{\sigma(X_s)}N( \hat u(X_s),X_s)
+\frac{1}{2}\,\abs{\sigma(X_s)^T\nabla\hat  f(X_s)}^2
\\-\frac{1}{2}\, \text{tr}\,(\sigma(X_s)\sigma(X_s)^T
\,\nabla^2  \hat f(X_s)\br)\,ds
-\int_0^te^{\hat f(X_0)-\hat f(X_s)} \nabla\hat  f(X_s)^T\sigma(X_s)\,d\hat W_s
\\=1+
\int_0^t e^{\hat f(X_0)-\hat f(X_s)}\bl(
\breve H(X_s;\lambda, \mathbf{0},\hat u)-F(\hat\lambda)\br)\,ds\\-
\int_0^te^{\hat f(X_0)-\hat f(X_s)} \nabla\hat  f(X_s)^T\sigma(X_s)\,d\hat W_s\,,
\end{multline*}
where $\mathbf 0$ stands for the zero function.
Let 
\begin{equation*}
\hat\tau_R=\inf\{t\ge0:\, \abs{X_t}>R\}\,,
  \end{equation*}
 where $R>0$\,.
Since $\bl(\int_0^{t\wedge\hat \tau_R}e^{\hat f(X_0)-\hat f(X_s)}
\nabla\hat  f(X_s)^T\sigma(X_s)\,d\hat W_s\,,t\ge0\br)$ 
is a martingale  with respect to $\hat{ \mathbf
  P}$\,, 
\begin{equation*}
  \hat{\mathbf E} e^{\hat f(X_0)-\hat f(X_{t\wedge \hat\tau_R})}
=  1
+\hat{\mathbf E}
\int_0^{t\wedge\hat\tau_R} e^{\hat f(X_0)-\hat f(X_s)}\bl(
\breve H(X_s;\lambda, \mathbf 0,\hat u)-F(\hat\lambda)
\br)\,ds\,.
\end{equation*}
By \eqref{eq:34} (with $\tau=\infty$) and \eqref{eq:39}, there exists $K>0$ such that
$\breve H(x;\lambda, \mathbf 0,\hat u)-F(\hat\lambda)<0$ if $\abs{x}>K$\,.
Therefore,
\begin{equation*}
  \hat{\mathbf E} e^{\hat f(X_0)-\hat f(X_{t\wedge \hat\tau_R})}
\le  1
+\sup_{\abs{x}\le K}e^{2\abs{\hat f(x)}}
\sup_{\abs{x}\le K}\bl(\abs{
\breve H(x;\lambda, \mathbf 0,\hat u)-F(\hat\lambda)}\br)t\,,
\end{equation*}
so, by Fatou's lemma, 
\begin{equation*}
  \hat{\mathbf E} e^{\hat f(X_{0})-\hat f(X_t)}\le1
+\sup_{\abs{x}\le K}e^{2\abs{\hat f(x)}}
\sup_{\abs{x}\le K}\bl(\abs{
\breve H(x;\lambda, \mathbf 0,\hat u)-F(\hat\lambda)}\br)t\,,
\end{equation*}
which implies, by \eqref{eq:40}, that
\begin{equation*}
  \limsup_{t\to\infty}\frac{1}{t}\,\ln
\mathbf E e^{-t\hat\lambda L^{\hat\pi}_t}\le F(\hat\lambda)\,.
\end{equation*}
Hence,
\begin{equation*}
  \limsup_{t\to\infty}\frac{1}{t}\,\ln \mathbf P(L^{\hat\pi}_t\le q)
\le \hat\lambda q+ \limsup_{t\to\infty}\frac{1}{t}\,\ln
\mathbf E e^{-t\hat\lambda L^{\hat\pi}_t}\le
\hat\lambda q+ F(\hat\lambda)\,.
\end{equation*}
\qed

\bibliographystyle{spmpsci}
\def\cprime{$'$} \def\cprime{$'$} \def\cprime{$'$} \def\cprime{$'$}
  \def\cprime{$'$} \def\polhk#1{\setbox0=\hbox{#1}{\ooalign{\hidewidth
  \lower1.5ex\hbox{`}\hidewidth\crcr\unhbox0}}} \def\cprime{$'$}
  \def\cprime{$'$} \def\cprime{$'$} \def\cprime{$'$} \def\cprime{$'$}
  \def\cprime{$'$}

\end{document}